\documentclass[nointlimits,11pt,oneside]{amsart}
\usepackage{amssymb,natbib,cases,enumerate}
\usepackage{color}
\usepackage{hyperref}
\usepackage{mathscinet}
\usepackage[%
	a4paper,
	total={16cm,23cm},
	left=1cm, top=3cm,
	marginparsep=2pt]
{geometry}

\theoremstyle{plain}
\newtheorem{theorem}{Theorem}[section]

\newtheorem{proposition}[theorem]{Proposition}
\theoremstyle{definition}
\newtheorem{remark}[theorem]{Remark}

\numberwithin{equation}{section}

\def\vr{\varrho}
\def\s{\sigma}

\def\M{\mathcal M}
\def\Mpl{\mathcal M_+}
\def\N{\mathbb N}
\def\R{\mathbb R}
\def\rn{\mathbb R^n}

\def\esssup{\operatornamewithlimits{ess\,\sup}}

\def\A{{\mathbb A}}
\def\B{{\mathbb B}}
\def\C{{\mathbb C}}

\def\d{{\fam0 d}}
\def\L1loc{L^1_{\text{loc}}}

\begin{document}

\title{Embeddings of homogeneous Sobolev spaces on the entire space}
\author{Zden\v ek Mihula}

\address{Zden\v ek Mihula, Charles University, Faculty of Mathematics and
Physics, Department of Mathematical Analysis, Sokolovsk\'a~83, 186~75 Praha~8, Czech Republic}
\email{mihulaz@karlin.mff.cuni.cz}
\urladdr{0000-0001-6962-7635}


\thanks{This research was supported by the grant P201-18-00580S of the Grant Agency of the Czech Republic, by the grant SVV-2017-260455, and by Charles University Research program No.~UNCE/SCI/023.}

\begin{abstract}
We completely characterize the validity of the inequality $\|u\|_{Y(\rn)}\leq C \|\nabla^m u\|_{X(\rn)}$, where $X$ and $Y$ are rearrangement-invariant spaces, by reducing it to a~considerably simpler one-dimensional inequality. Furthermore, we fully describe the optimal rearrangement-invariant space on either side of the inequality when the space on the other side is fixed. We also solve the same problem within the environment in which the competing spaces are Orlicz spaces. A~variety of examples involving customary function spaces suitable for applications is also provided.
\end{abstract}

\date{\today}

\maketitle

\setcitestyle{numbers}
\bibliographystyle{plainnat}


\section{Introduction}
The celebrated Gagliardo--Nirenberg--Sobolev inequality, which was proved for $p\in(1,n)$ by Sobolev and for $p=1$ by Gagliardo and Nirenberg independently, tells us that there exists a positive constant $C$ such that
\begin{equation}\label{intr:gnsineq}
\|u\|_{L^{p^*}(\rn)}\leq C\|\nabla u\|_{L^p(\rn)}\quad\text{for each $u\in W^{1,p}(\rn)$},
\end{equation}
where $n\in\N,\ n\geq2$, $p\in[1,n)$ and $p^*=\frac{np}{n-p}$. Here $W^{1,p}(\rn)$ stands for the Sobolev space of all weakly differentiable functions $u$ on $\rn$ that together with their gradients belong to $L^p(\rn)$. This result and its various modifications is classical and can be found in a wide variety of literature (e.g.~\citep{MR2424078, MR0102740, Mabook, MR0109940, Sob38, MR1125990, ziemerweakly}).
The Gagliardo--Nirenberg--Sobolev inequality and its consequences proved undoubtedly to be indispensable tools for analysis of partial differential equations, harmonic analysis and other fields of mathematics. The inequality \eqref{intr:gnsineq} was refined by Peetre (\citep{MR0221282}), utilizing the convolution inequality of O'Neil's (\citep{MR0146673}), to
\begin{equation}\label{intr:peetreoneilineq}
\|u\|_{L^{p^*, p}(\rn)}\leq C\|\nabla u\|_{L^p(\rn)}\quad\text{for each $u\in W^{1,p}(\rn)$},
\end{equation}
where $L^{p^*, p}(\rn)$ is a Lorentz space (for the definition of Lorentz spaces, see Section~\ref{sec:prel}). The inequality \eqref{intr:peetreoneilineq} is a substantial improvement of \eqref{intr:gnsineq} because the Lorentz space $L^{p^*, p}(\rn)$ is strictly smaller than the Lebesgue space $L^{p^*}(\rn)$. By iteration arguments one can also derive inequalities similar to the inequalities above where the first order gradient on the right-hand side is replaced by $m-$th order gradient, where $m>1$.

Theory as well as applications shows that finer scales of function spaces are indeed needed and so subtler forms of the Gagliardo--Nirenberg--Sobolev inequality involving more general function spaces are of great interest in mathematical analysis and its applications (e.g.~\citep{MR1814973, MR1670307, BW, MR0216286, Zhikov-1987}).

In this paper we focus on inequalities in which norms of scalar functions of several variables are compared to norms of their gradients from a broader perspective. It is known that Lebesgue spaces as well as more general Lorentz spaces are special instances of the so-called rearrangement-invariant spaces, which are, loosely speaking, Banach spaces of functions whose norms depend merely on the size of functions. We will consider inequalities taking the form
\begin{equation}\label{intr:homoemb}
\|u\|_Y\leq C \|\nabla^m u\|_X\quad\text{for each $u\in V_0^m X(\rn)$},
\end{equation}
where $C$ si a~positive constant independent of $u$, $m\in\N$, $1\leq m < n$, $X$ and $Y$ are rearrangement-invariant spaces over $\rn$ and $V_0^m X(\rn)$ is a vector space of all $m$-times weakly differentiable functions on $\rn$ whose $m$-th order gradients belong to $X$ and whose derivatives up to order $m-1$ have ``some decay at infinity''. In some sense the most general condition that ensures such decay is to assume that $|\{x\in\rn\colon|\nabla^k u(x)|>\lambda\}|<\infty$ for each $\lambda>0$ and $k=0,1,\dots,m-1$. This means that any integrability assumptions on $u$ itself and its lower-order derivatives are not needed and it is enough to assume that they ``decay at infinity'', albeit arbitrarily slowly. Precise definitions as well as other theoretical background needed in this paper are provided in Section~\ref{sec:prel}. We note that embeddings of Sobolev spaces on $\rn$ in the class of rearrangement-invariant spaces were studied in \citep{ACPS, MR2307145} but with the right-hand side involving the full gradient (that is, derivatives of all orders). Therefore, the problem studied there is essentially different from the one investigated in this paper as our right-hand side involves only the $m$-th order gradient, see \eqref{intr:homoemb}.

The main results regarding the inequality \eqref{intr:homoemb} are contained in Section~\ref{sec:mainresults}. We prove, among other things, so-called \emph{reduction principle} for the inequality \eqref{intr:homoemb}. This reduction principle (see Theorem~\ref{thm:reductionprinciple}) reveals that the inequality \eqref{intr:homoemb} is, in fact, equivalent to a one-dimensional inequality involving a~weighted Hardy-type operator. Moreover, for a fixed rearrangement-invariant space $X$ over $\rn$, we fully characterize the best possible (i.e.~the smallest possible) rearrangement-invariant space $Y$ over $\rn$ that renders \eqref{intr:homoemb} true (see Theorem~\ref{thm:optimalrange}). Complementing this result, we also answer the opposite question what the best possible (i.e.~the largest possible) rearrangement-invariant space $X$ over $\rn$ that renders \eqref{intr:homoemb} true for a fixed rearrangement-invariant space $Y$ over $\rn$ is (see Theorem~\ref{thm:optimaldomain}). The results presented in Section~\ref{sec:mainresults} are then proved in Section~\ref{sec:proofs}. We note that reduction principles have been successfully applied before, see e.g.~\cite{CPS, EKP, T2}.

The general results presented in Section~\ref{sec:mainresults} may be considered somewhat complicated from the point of view of applications in partial differential equations or harmonic analysis. For this reason, we provide a variety of concrete examples of optimal spaces in \eqref{intr:homoemb} for customary function spaces of particular interest in applications in Section~\ref{sec:riexamples}. These examples include, in particular, Lebesgue spaces, Lorentz spaces, Orlicz spaces, or Zygmund classes. For instance, these examples reveal that not only is the result of Peetre's (i.e.~\eqref{intr:peetreoneilineq}) better than \eqref{intr:gnsineq}, but it cannot, in fact, be improved. More precisely, the Lorentz space $L^{p^*,p}$ is the smallest possible rearrangement-invariant space on the left-hand side of \eqref{intr:peetreoneilineq} that renders the inequality true. Similar results are provided for other situations too.

Although the class of rearrangement-invariant spaces is very rich and contains many customary function spaces, it is sometimes useful in applications to work within a narrower class of function spaces. A typical example of such a class is that of Orlicz spaces, which is an irreplaceable tool for analysing partial differentiable equations having a non-polynomial growth (e.g.~\citep{MR3328350, DONALDSON1971507, VUILLERMOT1982327}). This motivates Section~\ref{sec:orliczexamples}. We investigate the inequality
\begin{equation}\label{intr:homoembOrlicz}
\|u\|_{L^B}\leq C\|\nabla^m u\|_{L^A}\quad\text{for each $u\in V_0^m L^A(\rn)$},
\end{equation}
where $L^A$ and $L^B$ are Orlicz spaces over $\rn$. We characterize optimal Orlicz spaces on either side of the inequality above while the Orlicz space on the opposite side is fixed (see Theorem~\ref{thm:optimalorliczrange} and Theorem~\ref{thm:optimalorliczdomain}) and we also provide a reduction principle for the inequality \eqref{intr:homoembOrlicz} (see Theorem~\ref{thm:reductionprincipleOrlicz}). To illustrate the general situation some concrete examples of optimal Orlicz spaces in \eqref{intr:homoembOrlicz} are also provided in Section~\ref{sec:orliczexamples}. In particular, these examples show that the Lebesgue space $L^{p^*}$ is the smallest possible Orlicz space on the left-hand side of the inequality \eqref{intr:gnsineq} that renders the inequality true. We stress that the crucial difference between Section~\ref{sec:orliczexamples} and Section~\ref{sec:mainresults} is that, in Section~\ref{sec:orliczexamples}, we look for optimal spaces that stay in the narrower class of Orlicz spaces. Although Orlicz spaces are particular instances of rearrangement-invariant spaces and so one is entitled to use the results from Section~\ref{sec:mainresults}, there is no guarantee that resulting optimal rearrangement-invariant spaces are Orlicz spaces themselves. Finally, we note that the inequality \eqref{intr:homoembOrlicz} was partially studied in \citep{MR2073127}. However, only the first order version (i.e.~$m=1$) of the inequality was studied there and optimality of Orlicz spaces only on the left-hand side of the inequality was considered.

\section{Preliminaries}\label{sec:prel}
In this section we collect all the background material that will be used in the
paper. We start with the operation of the nonincreasing rearrangement of a
measurable function.

Throughout this section, let $(R,\mu)$ be a $\sigma$-finite nonatomic measure space.
We set
\[
	\M(R,\mu)= \{f\colon \text{$f$ is $\mu$-measurable function on $R$ with values in $[-\infty,\infty]$}\},
\]
\[
	\M_0(R,\mu)= \{f \in \M(R,\mu)\colon f \ \text{is finite}\ \mu\text{-a.e.\ on}\ R\}
\]
and
\[
\M_+(R,\mu)= \{f \in \M(R,\mu)\colon f \geq 0\}.
\]
The \textit{nonincreasing rearrangement} $f^* \colon  (0,\infty) \to [0, \infty ]$ of a function $f\in \M(R,\mu)$  is
defined as
\[
f^*(t)=\inf\{\lambda\in(0,\infty)\colon|\{s\in R\colon|f(s)|>\lambda\}|\leq t\},\ t\in(0,\infty).
\]
The \textit{maximal nonincreasing rearrangement} $f^{**} \colon (0,\infty) \to [0, \infty ]$ of a function $f\in \M(R,\mu)$  is
defined as
\[
f^{**}(t)=\frac1t\int_0^ t f^{*}(s)\,\d s,\ t\in(0,\infty).
\]
If $|f|\leq |g|$ $\mu$-a.e.\ in $R$, then $f^*\leq g^*$.
The operation $f\mapsto f^*$ does not preserve sums or products of functions,
and is known not to be subadditive. The lack of subadditivity of the operation
of taking the nonincreasing rearrangement is, up to some extent, compensated
by the following fact (\cite[Chapter~2,~(3.10)]{BS}): for every
$t\in(0,\infty)$ and every $f,g\in\mathcal M(R,\mu)$, we have
\begin{equation*}
\int_0^ t(f +g)^{*}(s)\,\d s
\leq
\int_0^ tf^{*}(s)\,\d s + \int_0^ tg
^{*}(s)\,\d s.
\end{equation*}
This inequality can be also written in the form
\begin{equation}\label{E:subadditivity-of-doublestar-a}
(f+g)^{**}\leq f^{**}+g^{**}.
\end{equation}
Another important property of rearrangements is the \textit{Hardy-Littlewood
inequality} (\citep[Chapter~2, Theorem~2.2]{BS}), which asserts that if $f, g \in\M(R,\mu)$,
then
\begin{equation}\label{E:HL}
\int _R |fg| \,\d\mu \leq \int _0^{\infty} f^*(t) g^*(t)\,\d t.
\end{equation}

If $(R,\mu)$ and $(S,\nu)$ are two (possibly different) $\sigma$-finite measure spaces, we say that functions $f\in \M(R,\mu)$ and $g\in\M(S,\nu)$ are \textit{equimeasurable}, and write $f\sim g$, if $f^*=g^*$ on $(0,\infty)$.

A functional $\vr\colon  \M _+ (R,\mu) \to [0,\infty]$ is called a \textit{Banach function norm} if, for all $f$, $g$ and
$\{f_j\}_{j\in\N}$ in $\M_+(R,\mu)$, and every $\lambda \geq0$, the
following properties hold:
\begin{enumerate}[(P1)]
\item $\vr(f)=0$ if and only if $f=0$;
$\vr(\lambda f)= \lambda\vr(f)$; $\vr(f+g)\leq \vr(f)+ \vr(g)$ (the \textit{norm axiom});
\item $  f \le g$ a.e.\  implies $\vr(f)\le\vr(g)$ (the \textit{lattice axiom});
\item $  f_j \nearrow f$ a.e.\ implies
$\vr(f_j) \nearrow \vr(f)$ (the \textit{Fatou axiom});
\item $\vr(\chi_E)<\infty$ \ for every $E\subseteq R$ of finite measure (the \textit{nontriviality axiom});
\item  if $E$ is a subset of $R$ of finite measure, then $\int_{E} f\,\d\mu \le C_E
\vr(f)$ for a positive constant $C_E$, depending possibly on $E$ and $\vr$ but independent of $f$ (the \textit{local embedding in $L^1$}).
\end{enumerate}
If, in addition, $\vr$ satisfies
\begin{itemize}
\item[(P6)] $\vr(f) = \vr(g)$ whenever
$f^* = g^ *$(the \textit{rearrangement-invariance axiom}),
\end{itemize}
then we say that $\vr$ is a
\textit{rearrangement-invariant norm}.

If $\vr$ is a rearrangement-invariant norm, then the collection
\[
X=X({\vr})=\{f\in\M(R,\mu)\colon \vr(|f|)<\infty\}
\]
is called a~\textit{rearrangement-invariant space}, sometimes we shortly write just an~\textit{r.i.~space}, corresponding to the norm $\vr$. We shall write $\|f\|_{X}$ instead of $\vr(|f|)$. Note that the quantity $\|f\|_{X}$ is defined for every $f\in\M(R,\mu)$, and
\[
f\in X\quad\Leftrightarrow\quad\|f\|_X<\infty.
\]

With any rearrangement-invariant function norm $\vr$, there is associated another functional, $\vr'$, defined for $g \in  \M_+(R,\mu)$ as
\[
\vr'(g)=\sup\left\{\int_{R} fg\,\d\mu\colon f\in\M_+(R,\mu),\ \vr(f)\leq 1\right\}.
\]
It turns out that  $\vr'$ is also a
rearrangement-invariant norm, which is called
the~\textit{associate norm} of $\vr$. Moreover, for every rearrangement-invariant norm $\vr$ and every $f\in\Mpl(R,\mu)$, we have (see~\citep[Chapter~1, Theorem~2.9]{BS})
\[
\vr(f)=\sup\left\{\int_{R}fg\,\d\mu\colon g\in\M_+(R,\mu),\ \vr'(f)\leq 1\right\}.
\]
By~\citep[Chapter~2, Proposition~4.2]{BS} we, in fact, have
\[
\vr'(g)=\sup\left\{\int_0^{\mu(R)} f^*(t)g^*(t)\,\d t\colon f\in\M(R,\mu),\ \vr(f)\leq 1\right\}
\]
and
\[
\vr(f)=\sup\left\{\int_0^{\mu(R)}f^*(t)g^*(t)\,\d t\colon g\in\M(R,\mu),\ \vr'(f)\leq 1\right\}.
\]

If $\vr$ is a~rearrangement-invariant norm, $X=X({\vr})$ is the rearrangement-invariant space determined by $\vr$, and $\vr'$ is the associate norm of $\vr$, then the function space $X({\vr'})$ determined by $\vr'$ is called the \textit{associate space} of $X$ and is denoted by $X'$. We always have $(X')'=X$ (see~\citep[Chapter~1, Theorem~2.7]{BS}), and we shall write $X''$ instead of $(X')'$. Furthermore,
the \textit{H\"older inequality}
\[
\int_{R}|fg|\,\d \mu\leq\|f\|_{X}\|g\|_{X'}
\]
holds for every $f,g\in \M(R,\mu)$.

We say that a~rearrangement-invariant space $X$ is \emph{embedded into} a~rearrangement-invariant space $Y$, and we write
\begin{equation}\label{pre:embedding}
X\hookrightarrow Y,
\end{equation}
if $X\subseteq Y$ and the inclusion is continuous, that is, there exists a positive constant $C$ such that
\begin{equation*}
\|f\|_Y\leq C\|f\|_X\quad\text{for each }f\in X.
\end{equation*}
However, it turns out that \eqref{pre:embedding} holds if and only if $X\subseteq Y$ (\citep[Chapter~1, Theorem~1.8]{BS}).

Another important property (see \citep[Chapter~1, Proposition~2.10]{BS}), which we shall exploit several times, is that \eqref{pre:embedding} holds if and only if
\begin{equation}\label{pre:embeddingdual}
Y'\hookrightarrow X'.
\end{equation}
Moreover, if \eqref{pre:embedding} holds, then \eqref{pre:embeddingdual} holds in fact with the same embedding constant, and vice versa.

For every rearrangement-invariant space $X$ over the measure space $(R,\mu)$, there exists a~unique rearran\-gement-invariant space $X(0,\mu(R))$ over the interval $(0,\mu(R))$ endowed with the one-dimensional Lebesgue measure such that $\|f\|_X=\|f^*\|_{X(0,\mu(R))}$. This space is called the~\textit{representation space} of $X$. This follows from the Luxemburg representation theorem (see \citep[Chapter~2, Theorem~4.10]{BS}). Throughout this paper, the representation space of a rearrangement-invariant space $X$ will be denoted by $X(0,\mu(R))$. It will be useful to notice that when $R=(0,\infty)$  and $\mu$ is the Lebesgue measure, then every $X$ over $(R,\mu)$ coincides with its representation space.

If $\vr$ is a~rearrangement-invariant norm and $X=X({\vr})$ is the
rearrangement-invariant space determined by $\vr$, we define its
\textit{fundamental function}, $\varphi_X$, by
\begin{equation*}
\varphi_X(t)=\varrho(\chi_E),\ t\in[0,\mu(R)),
\end{equation*}
where $E\subseteq R$ is such that $\mu(E)=t$. The property (P6) of rearrangement-invariant norms and the fact that $\chi_E^* = \chi_{[0, \mu(E))}$ guarantee that
the fundamental function is well defined. Moreover, one has
\begin{equation*}
	\varphi_X(t)\varphi_{X'}(t)=t
	\qquad \textup{for every}\ t\in[0,\mu(R)).
\end{equation*}
For each $a\in(0,\infty)$, let $D_a$ denote the \textit{dilation operator}
defined on every nonnegative measurable function $f$ on $(0,\infty)$ by
\[
(D_af)(t)=f(at), \ t\in(0,\infty).
\]
The dilation operator $D_a$ is bounded on every rearrangement-invariant space over
$(0,\infty)$; hence, in  particular, on the representation space of any
rearrangement-invariant space over an~arbitrary $\sigma$-finite nonatomic measure space. More
precisely, if $X$ is any given rearrangement-invariant space over $(0,\infty)$
with respect to the one-dimensional Lebesgue measure, then we have
\[
\|D_af\|_{X}\leq C\|f\|_X\quad \textup{for every}\ f\in X,
\]
with some constant $C$, $0<C\leq\max\{1,\frac1{a}\}$, independent
of $f$. For more details, see~\citep[Chapter~3, Proposition~5.11]{BS}.

Basic examples of function norms are those associated with the standard
Lebesgue spaces $L^ p$. For $p\in(0,\infty]$, we define the functional $\vr_p$
by
\[
\vr_p(f)=\|f\|_p=
\begin{cases}
\left(\int_Rf^ p\,\d\mu\right)^{\frac1p}, &\quad0<p<\infty,\\
\esssup_{R}f,&\quad p=\infty,
\end{cases}
\]
for $f \in \M_+(R,\mu)$. If $p\in[1,\infty]$, then $\vr_p$ is a rearrangement-invariant function norm.

If $0< p,q\le\infty$, we define the functional
$\vr_{p,q}$ by
\[
\vr_{p,q}(f)=\|f\|_{p,q}=
\left\|s^{\frac{1}{p}-\frac{1}{q}}f^*(s)\right\|_{q}
\]
for $f \in \M_+(R,\mu)$. The set $L^{p,q}$, defined as the collection of all $f\in\M(R,\mu)$ satisfying $\vr_{p,q}(|f|)<\infty$, is called a~\textit{Lorentz space}. If $1<p<\infty$ and $1\leq q\leq\infty$, $p=q=1$, or $p=q=\infty$, then $\vr_{p,q}$ is equivalent to a~rearrangement-invariant function norm in the sense that there exists a~rearrangement-invariant norm $\sigma$ and a~constant $C$, $0<C<\infty$, depending on $p,q$ but independent of $f$, such that
\[
C^{-1}\sigma(f)\leq \vr_{p,q}(f)\leq C\sigma(f).
\]
As a~consequence, $L^{p,q}$ is considered to be a~rearrangement-invariant space for the above specified cases of $p,q$ (see~\citep[Chapter~4]{BS}). If either $0<p<1$ or $p=1$ and $q>1$, then $L^{p,q}$ is a~quasi-normed space. If $p=\infty$ and $q<\infty$, then $L^{p,q}=\{0\}$. For every $p\in[1,\infty]$, we have $L^{p,p}=L^{p}$. Furthermore, if $p,q,r\in(0,\infty]$ and $q\leq r$, then the inclusion $L^{p,q}\subset L^{p,r}$ holds.

If $\A=[\alpha_0,\alpha_{\infty}]\in\R^ 2$ and $t\in\R$, then we shall use the notation $\A+t=[\alpha_0+t,\alpha_{\infty}+t]$ and $t\A=[t\alpha_0, t\alpha_{\infty}]$.

Let $0<p,q\le\infty$, $\A=[\alpha_0,\alpha_{\infty}]\in\R^ 2$ and $\B=[\beta_0,\beta_{\infty}]\in\R^ 2$. Then we define the
functionals $\vr_{p,q;\A}$ and $\vr_{p,q;\A,\B}$ on $\M_+(R,\mu)$ by
\[
\vr_{p,q;\A}(f)=
\left\|t^{\frac{1}{p}-\frac{1}{q}}\ell^{\A}(t) f^*(t)\right\|_{L^ q(0,\infty)}
\]
and
\[
\vr_{p,q;\A,\B}(f)=
\left\|t^{\frac{1}{p}-\frac{1}{q}}\ell^{\A}(t)\ell\ell^{\B}(t) f^*(t)\right\|_{L^ q(0,\infty)},
\]
where
\[
\ell^{\A}(t)=
\begin{cases}
(1-\log t)^{\alpha_0}, &\quad t\in(0,1),\\
(1+\log t)^{\alpha_{\infty}}, &\quad t\in[1,\infty),
\end{cases}
\]
and
\[
\ell\ell^{\B}(t)=
\begin{cases}
(1+\log (1-\log t))^{\beta_0}, &\quad t\in(0,1),\\
(1+\log (1+\log t))^{\beta_{\infty}}, &\quad t\in[1,\infty).
\end{cases}
\]
The set $L^{p,q;\A}$, defined as the collection of all $f\in\M(R,\mu)$ satisfying $\vr_{p,q;\A}(|f|)<\infty$, is called a~\textit{Lorentz--Zygmund space}, and the set $L^{p,q;\A,\B}$, defined as the collection of all $f\in\M_+(R,\mu)$ satisfying $\vr_{p,q;\A,\B}(|f|)<\infty$, is called a~\textit{generalized Lorentz--Zygmund space}. The functions of the form $\ell^{\A}$, $\ell\ell^{\B}$ are called \textit{broken logarithmic functions}. It can be shown (\citep[Theorem~7.1]{OP}) that the functional $\vr_{p,q;\A}$ is equivalent to a rearrangement-invariant function norm if and only if
\begin{equation}\label{thm:optimalrangeexamples-condonglztobebfs}\begin{cases}
&p=q=1,\ \alpha_0\geq0,\ \alpha_\infty\leq0\ \text{or}\\
&p\in(1,\infty)\ \text{or}\\
&p=\infty,\ q\in[1,\infty),\ \alpha_0 + \frac1{q} < 0\ \text{or}\\
&p=q=\infty,\ \alpha_0\leq0.
\end{cases}
\end{equation}

The spaces of this type proved to be quite useful since they provide a common roof for many customary spaces. These include not only Lebesgue spaces and Lorentz spaces, by taking $\A=[0,0]$, but also all types of exponential and logarithmic Zygmund classes, and also the spaces discovered independently by Maz'ya (in a~somewhat implicit form involving capacitary estimates~\citep[pp.~105 and~109]{Mabook}), Hansson~\citep{Ha} and Br\'ezis--Wainger~\citep{BW}, who used it to describe the sharp target space in a limiting Sobolev embedding (the spaces can be also traced in the works of Brudnyi~\citep{B} and, in a more general setting, Cwikel and Pustylnik~\citep{CP}). One of the benefits of using broken logarithmic functions consists in the fact that the underlying measure space can be considered to have either finite or infinite measure. For the detailed study of (generalized) Lorentz--Zygmund spaces we refer the reader to~\citep{EOP,EOP-broken,OP, FS}. In some examples in Section~\ref{sec:riexamples} we shall need more than two layers of logarithms. Such spaces are defined as a straightforward extension of the spaces defined above.

A convex, neither identically zero nor infinity, left-continuous function $A\colon [0, \infty ) \to [0, \infty ]$ vanishing
at $0$ is called a \emph{Young function}. Hence any Young function can be expressed in the form
\begin{equation}\label{pre:integralformofYoungfunction}
A(t) = \int _0^t a(s ) \d s \qquad \quad \text{for $t \geq 0$},
\end{equation}
 for some nondecreasing, left-continuous function $a\colon [0, \infty )\to
[0, \infty ]$. For a Young function $A$ we define the \emph{Luxemburg function norm} $\|\cdot \|_{L^A}$ as
\begin{equation*}
\|f\|_{L^A}= \inf \left\{ \lambda >0\colon\int_R A\left(
\frac{f(x)}{\lambda} \right) \d\mu(x) \leq 1 \right\},\ f \in {\Mpl(R,\mu)}.
\end{equation*}
The corresponding rearrangement-invariant space $L^A$ is called an \emph{Orlicz space}. In particular,
$L^A= L^p$ if $A(t)= t^p$ when $p \in [1, \infty )$ and $L^A= L^\infty$ if $A(t)=0$ for $t\in [0, 1]$ and
$A(t) = \infty$ for $t>1$.

The associate space of an Orlicz space $L^A$ is equivalent to another Orlicz space $L^{\widetilde{A}}$ where $\widetilde{A}$ is the \emph{Young conjugate function} of $A$, which is a Young function again, defined by
\begin{equation*}
\widetilde{A}(t) = \sup\limits_{0\leq s<\infty}\left(st-A(s)\right).
\end{equation*}

We say that a Young function $A$ \emph{dominates} a Young function $B$ \emph{near zero} or \emph{near infinity} if there exist positive constants $c$ and $t_0$ such that
\begin{equation*}
B(t)\leq A(ct)\quad\text{for all $t\in[0,t_0]$ or for all $t\in[t_0,\infty)$, respectively.}
\end{equation*}
We say that two Young functions $A$ and $B$ are \emph{equivalent near zero} or \emph{near infinity} if they dominate each other near zero or near infinity, respectively. We say that they are \emph{equivalent globally} if they are equivalent near zero and equivalent near infinity simultaneously.

If, for a nonnegative measurable function $F$ on $(0,\infty)$, there exists $t_0>0$ such that $\int_0^{t_0}F(s)\,\d s<\infty$ or $\int_{t_0}^\infty F(s)\,\d s<\infty$, respectively, we shortly write that
\begin{equation*}
\int_0 F(s)\,\d s<\infty\text{ or }\int^\infty F(s)\,\d s<\infty\text{, respectively.}
\end{equation*}

If $A$ is a Young function, we define the function $h_A\colon(0,\infty)\rightarrow[0,\infty)$ by
\begin{equation*}
h_A(t)=\sup\limits_{0<s<\infty}\frac{A^{-1}(st)}{A^{-1}(s)},\ t>0,
\end{equation*}
and we set
\begin{equation}\label{thm:optimalorliczdomain:Boydlow}
i_A = \sup\limits_{1<t<\infty}\frac{\log t}{\log h_A(t)}
\end{equation}
and
\begin{equation}\label{thm:optimalorliczdomain:Boydupp}
I_A = \inf\limits_{0<t<1}\frac{\log t}{\log h_A(t)}.
\end{equation}
The quantities $i_A$ and $I_A$ are called the \emph{lower Boyd index of $A$} and the \emph{upper Boyd index of $A$}, respectively, and it can be shown that $1\leq i_A\leq I_A\leq\infty$, $i_A=\lim\limits_{t\rightarrow\infty}\frac{\log t}{\log h_A(t)}$ and $I_A=\lim\limits_{t\rightarrow0^+}\frac{\log t}{\log h_A(t)}$. We refer the interested reader to \citep{MR0126722, MR1113700} for more details on Orlicz spaces and to \citep{BS, MR0212512, MR0306887} for more details on Boyd indices.

A common extension of  Orlicz and Lorentz spaces is provided by the
family of \emph{Orlicz-Lorentz spaces}. Given $p\in (1, \infty)$, $q\in
[1, \infty )$ and  a Young function $A$ such that
\begin{equation}\label{pre:conditiononAforOrliczLorentz}
\int^\infty \frac{A(t)}{t^{1+p}}\,\d t < \infty,
\end{equation}
we denote by $\|\cdot \|_{L(p,q,A)}$ the Orlicz-Lorentz
rearrangement-invariant function norm defined as
\begin{equation}\label{pre:orliczlorentznormdef}
\|f\|_{L(p,q,A)} = \left\|t^{-\frac1{p}} f^*(t^{\frac1{q}})\right\|_{L^A(0, \mu(R))},\quad f \in\Mpl(R,\mu).
\end{equation}
The fact that \eqref{pre:orliczlorentznormdef} actually defines  a  rearrangement-invariant function norm follows from simple variants in the proof
of \citep[Proposition 2.1]{MR2073127}. We denote by $L(p,q,A)$ the Orlicz-Lorentz space associated with the rearrangement-invariant function norm $\|\cdot
\|_{L(p,q, A)}$. Note that the class of Orlicz-Lorentz spaces includes (up to equivalent norms) the Orlicz spaces and various instances of Lorentz and Lorentz-Zygmund spaces.

In what follows we shortly denote the Lebesgue measure of a measurable set $E$ by $|E|$.

We shall work with \emph{Sobolev-type spaces built upon rearrangement-invariant spaces}. If $m\in\N$ and $u$ is a $m$-times weakly differentiable function on $\rn$, we denote by $\nabla^k u$, for $k\in\{0, 1,\dots, m\}$, the vector of all weak derivatives of order $k$ of $u$, where $\nabla^0 u = u$. If $X$ is a rearrangement-invariant space over $\rn$, we define spaces $V^m X(\rn)$ and $V_0^m X(\rn)$ by
\begin{align*}
V^m X(\rn) &=\{u\colon \rn\rightarrow\R\colon\text{$u$ is $m$-times weakly differentiable and }|\nabla^m u|\in X\},\\
V_0^m X(\rn) &=\{u\in V^m X(\rn)\colon|\{x\in\rn\colon |\nabla^ku(x)|>\lambda\}|<\infty\ \text{for $k\in\{0,1,\dots,m - 1\}$ and $\lambda > 0$}\} .
\end{align*}
We stress the fact that, for a function from $V^m X(\rn)$, only its $m$-th order derivatives are required to be elements of $X$, whereas there are no assumptions imposed on its derivatives of lower orders. The derivatives of lower orders are not required to have any regularity, we merely assume that they exist. We also write $\|\nabla^k u\|_X$ instead of $\||\nabla^k u|\|_X$ for the sake of brevity, where $|\nabla^k u|$ is the $\ell^1$-norm of the vector $\nabla^k u$.

Throughout the paper the convention that $\frac1{\infty}=0$ and $0\cdot \infty
=0$ is used without further explicit reference. We write $A\lesssim B$ when $A\leq \text{constant}\cdot B$ where the constant is independent of appropriate
quantities appearing in expressions $A$ and $B$. Similarly, we write $A\gtrsim B$ with the obvious meaning. We also write $A\approx B$ when $A\lesssim B$ and $A\gtrsim B$ simultaneously.

We say that a rearrangement-invariant space $Y$ over $\rn$ is the \emph{optimal target space} (within the class of rearrangement-invariant spaces) for a rearrangement-invariant space $X$ over $\rn$ in \eqref{intr:homoemb} if \eqref{intr:homoemb} is satisfied and whenever \eqref{intr:homoemb} is satisfied for another rearrangement-invariant space $Z$ over $\rn$ in place of $Y$, $Z$ is larger than $Y$, that is, $Y\hookrightarrow Z$. We say that a rearrangement-invariant space $X$ over $\rn$ is the \emph{optimal domain space} (within the class of rearrangement-invariant spaces) for a rearrangement-invariant space $Y$ over $\rn$ in \eqref{intr:homoemb} if \eqref{intr:homoemb} is satisfied and whenever \eqref{intr:homoemb} is satisfied for another rearrangement-invariant space $Z$ over $\rn$ in place of $X$, $Z$ is smaller than $X$, that is, $Z\hookrightarrow X$.

\section{Main results}\label{sec:mainresults}
Our first theorem characterizes when, for a given rearrangement-invariant space $X$ over $\rn$, there exists a rearrangement-invariant space $Y$ over $\rn$ that renders \eqref{intr:homoemb} true by a condition on the associate space of $X$, and if the condition is satisfied, it provides a description of the optimal target space for $X$.
\begin{theorem}\label{thm:optimalrange}
Assume that $m < n$ and let $X$ be a rearrangement-invariant space over $\rn$ such that
\begin{equation}\label{thm:optimalrange:cond}
t^{\frac{m}{n} - 1}\chi_{(1,\infty)}(t)\in X'(0,\infty).
\end{equation}
Define the functional $\sigma_m$ by
\begin{equation}\label{thm:optimalrange:sigma}
	\s_m(f)=\|t^\frac{m}{n}f^{**}(t)\|_{X'(0,\infty)},\ f\in\Mpl(\rn).
\end{equation}
Then $\sigma_m$ is a~rearrangement-invariant norm and there exists a positive constant $C$, which depends on $m$ and on the dimension $n$ only, such that
\begin{equation}\label{thm:optimalrange:vnoreni}
\|u\|_{X^m}\leq C \|\nabla^m u\|_{X}\quad\text{for each $u\in V_0^m X(\rn)$}
\end{equation}
where $X^m=X^m(\sigma_m')$. Moreover, $X^m$ is the optimal \textup{(}smallest\textup{)} target space for $X$ in \eqref{intr:homoemb}.

Conversely, if~\eqref{thm:optimalrange:cond} is not true, then there does not exist any rearrangement-invariant space $Y$ for which~\eqref{intr:homoemb} is true at all.
\end{theorem}

We note that~\eqref{thm:optimalrange:cond} holds, for instance, for every $X=L^{p}$ with $p\in[1,\frac{n}{m})$ or for $X=L^{\frac{n}{m},1}$.

A somewhat surprising property of optimal target spaces is that they are stable under iteration (cf.~\citep[Theorem~1.5]{CPnew}, \citep[Theorem~5.7]{CPS}). This \emph{iteration principle} is the content of the following theorem.
\begin{theorem}\label{thm:iterationprinciple}
Let $k$ and $l$ be natural numbers such that $k+l < n$. Assume that $X$ is a~rearrangement-invariant space over $\rn$ such that \eqref{thm:optimalrange:cond} holds with $m = k + l$. Then \eqref{thm:optimalrange:cond} holds also with $m = k$, $t^{\frac{l}{n} - 1}\chi_{(1,\infty)}(t)\in (X^k)'(0,\infty)$ and the norms on $(X^k)^l$ and $X^{k+l}$ are equivalent, where the constants of the equivalence depend on $m$ and on the dimension $n$ only.
\end{theorem}

The following theorem establishes the \emph{reduction principle} for the inequality \eqref{intr:homoemb}.
\begin{theorem}\label{thm:reductionprinciple}
Assume that $m < n$ and let $X$ and $Y$ be rearrangement-invariant spaces over $\rn$. Then the following three inequalities are equivalent:
\begin{align}
\|u\|_Y&\leq C_1 \|\nabla^m u\|_X&\quad\text{for each $u\in V_0^m X(\rn)$};\label{thm:eq:reductionprinciple:embedding}\\
\left\|\int_t^\infty f(s)s^{\frac{m}{n} - 1}\, \d s\right\|_{Y(0,\infty)}&\leq C_2 \|f\|_{X(0,\infty)}&\quad\text{for each $f\in\Mpl(0,\infty)$};\label{thm:eq:equivalentformsofonedimensionalinequality:reduction}\\
\|t^{\frac{m}{n}}g^{**}(t)\|_{X'(0,\infty)}&\leq C_2 \|g\|_{Y'(0,\infty)}&\quad\text{for each $g\in\Mpl(0,\infty)$},\label{thm:eq:equivalentformsofonedimensionalinequality:reductiondual}
\end{align}
where the positive constants $C_1$ and $C_2$ depend on each other, on $m$ and on the dimension $n$ only.
\end{theorem}
In fact, the inequality \eqref{thm:eq:equivalentformsofonedimensionalinequality:reduction} is (and so are the other two inequalities) equivalent to the same inequality but restricted to nonincreasing functions only. More precisely, \eqref{thm:eq:equivalentformsofonedimensionalinequality:reduction} is equivalent to (with a possibly different positive constant $C$)
\begin{equation*}
\left\|\int_t^\infty f^*(s)s^{\frac{m}{n} - 1}\, \d s\right\|_{Y(0,\infty)}\leq C \|f\|_{X(0,\infty)}\quad\text{for each $f\in\Mpl(0,\infty)$}.
\end{equation*}
This equivalence is a special case of the general result that originated as a consequence (\citep[Corollary~9.8]{CPS}) of a more general principle established in~\citep[Theorem~9.5]{CPS} in connection with sharp higher-order Sobolev-type embeddings and its extension to unbounded intervals was given in~\citep[Theorem~1.1]{P}.
\begin{remark}\label{rem:relationshipbetweenembeddingandfractional}
There is an intimate connection between the inequality \eqref{intr:homoemb} and the fractional maximal operator $M_\gamma$, which is defined for a fixed $\gamma\in(0,n)$ and for a locally integrable function $f$ on $\rn$ by
\begin{equation*}
M_\gamma f(x)=\sup\limits_{Q\ni x}\frac1{|Q|^{1-\frac{\gamma}{n}}}\int_Q|f(y)|\,\d y,\ x\in\rn,
\end{equation*}
where the supremum is taken over all cubes $Q\subseteq\rn$ whose edges are parallel to the coordinate axes and that contain $x$. If $m < n$, then the inequality \eqref{intr:homoemb} is true for a pair of rearrangement-invariant spaces $X$ and $Y$ if and only if
\begin{equation*}
M_m\colon Y'\rightarrow X'
\end{equation*}
is bounded because it follows from the arguments used in the proof of \citep[Theorem~4.1]{EMMP} that $M_m\colon Y'\rightarrow X'$ is bounded if and only if \eqref{thm:eq:equivalentformsofonedimensionalinequality:reductiondual} is valid, which is equivalent to \eqref{intr:homoemb} by Theorem~\ref{thm:reductionprinciple}.
\end{remark}

Complementing Theorem~\ref{thm:optimalrange}, the following theorem characterizes when for a given rearrangement-invariant space $Y$ over $\rn$, there exists a rearrangement-invariant space $X$ over $\rn$ rendering \eqref{intr:homoemb} true by a condition on the fundamental function of the space $Y$, and if the condition is satisfied, it provides a description of the optimal domain space.
\begin{theorem}\label{thm:optimaldomain}
Assume that $m < n$ and let $Y$ be a rearrangement-invariant space over $\rn$ such that
\begin{equation}\label{thm:optimaldomain:cond}
\inf\limits_{1\leq t < \infty}\frac{t^{1-\frac{m}{n}}}{\varphi_{Y}(t)} > 0.
\end{equation}
Define the functional $\tau_m$ by
\begin{equation}\label{thm:optimaldomain:tau}
	\tau_m(f)=\sup_{\substack{h\sim f\\h\ge0}}
		\left\|\int_t^{\infty}h(s)s^{\frac{m}{n}-1}\,\d s\right\|_{Y(0,\infty)},
		\ f\in\Mpl(\rn),
\end{equation}
where the supremum is taken over all $h\in\Mpl(0,\infty)$ equimeasurable with $f$.
Then $\tau_m$ is a~rearrangement-invariant norm and there exists a positive constant $C$, which depends on $m$ and on the dimension $n$ only, such that
\begin{equation}\label{thm:optimaldomain:vnoreni}
\|u\|_{Y}\leq C \|\nabla^m u\|_{Y_m}\quad\text{for each $u\in V_0^m Y_m(\rn)$}
\end{equation}
where $Y_m=Y_m(\tau)$. Moreover, $Y_m$ is the optimal \textup{(}largest\textup{)} domain space for $Y$ in \eqref{intr:homoemb}.

Conversely, if~\eqref{thm:optimaldomain:cond} is not true, then there does not exist any rearrangement-invariant space $X$ for which~\eqref{intr:homoemb} is true at all.
\end{theorem}

The general description of the optimal domain norm given by \eqref{thm:optimaldomain:tau} is quite complicated. Fortunately, it can be simplified significantly in many customary situations. This is the content of the following statement, which follows from Theorems~4.2 and~4.7 in~\citep{EMMP}. We shall need the operator $T_\alpha$ defined for some fixed $\alpha\in(0,1)$ by
\begin{equation}\label{rem:simplificationofdomainnormsupopdef}
T_\alpha f(t)= t^{-\alpha}\sup_{t\leq s<\infty}s^\alpha f^*(s)\quad\text{for $t\in(0,\infty)$ and $f\in\M(0,\infty)$.}
\end{equation}

\begin{theorem}\label{thm:simplificationofdomainnorm}
Assume that $m < n$ and let $Y$ be a rearrangement-invariant space over $\rn$ such that the operator $T_{\frac{m}{n}}$ is bounded on $Y'(0,\infty)$. Then \eqref{thm:optimaldomain:cond} is satisfied and the rearrangement-invariant norm $\tau_m$ defined by \eqref{thm:optimaldomain:tau} is equivalent to the functional
\begin{equation}\label{thm:simplificationofdomainnormeq}
f\mapsto\left\|\int_t^{\infty}f^*(s)s^{\frac{m}{n}-1}\,\d s\right\|_{Y(0,\infty)},\ f\in\M_+(0,\infty).
\end{equation}
Conversely, if $T_{\frac{m}{n}}$ is not bounded on $Y'(0,\infty)$, then $\tau_m$ is not equivalent to the functional~\eqref{thm:simplificationofdomainnormeq}.
\end{theorem}
We finish this section by observing that Theorem~\ref{thm:simplificationofdomainnorm} can be applied, for example, to $Y=L^{p}$ with $p\in(\frac{n}{n-m},\infty]$ or to $Y=L^{\frac{n}{n-m},1}$.

\section{Proofs of main results}\label{sec:proofs}
We start off by proving the equivalence of \eqref{thm:eq:equivalentformsofonedimensionalinequality:reduction} and \eqref{thm:eq:equivalentformsofonedimensionalinequality:reductiondual}.
\begin{proposition}\label{prop:equivalentformsofonedimensionalinequality}
Assume that $m < n$ and let $X(0,\infty)$ and $Y(0,\infty)$ be rearrangement-invariant spaces over $(0,\infty)$. Then the following two inequalities (in fact with the same positive constants $C$) are equivalent:
\begin{align*}
\left\|\int_t^\infty f(s)s^{\frac{m}{n} - 1}\, \d s\right\|_{Y(0,\infty)}&\leq C \|f\|_{X(0,\infty)}\quad\text{for each $f\in\Mpl(0,\infty)$};\\
\|t^{\frac{m}{n}}g^{**}(t)\|_{X'(0,\infty)}&\leq C \|g\|_{Y'(0,\infty)}\quad\text{for each $g\in\Mpl(0,\infty)$}.
\end{align*}
\end{proposition}
\begin{proof}
The equivalence of these two inequalities follows from the definition of the associate norm because we have that
\begin{align*}
\sup_{\substack{\|f\|_{X(0,\infty)}\leq1\\f\geq0}}\left\|\int_t^\infty f(s)s^{\frac{m}{n} - 1}\, \d s\right\|_{Y(0,\infty)} &= \sup_{\substack{\|f\|_{X(0,\infty)}\leq1\\f\geq0}}\sup_{\substack{\|g\|_{Y'(0,\infty)}\leq1\\g\geq0}}\int_0^\infty g(t)\int_t^\infty f(s)s^{\frac{m}{n} - 1}\, \d s\, \d t\\
&= \sup_{\substack{\|f\|_{X(0,\infty)}\leq1\\f\geq0}}\sup_{\substack{\|g\|_{Y'(0,\infty)}\leq1\\g\geq0}}\int_0^\infty f(s) s^{\frac{m}{n} - 1}\int_0^s g(t)\,\d t\,\d s\\
&= \sup_{\substack{\|f\|_{X(0,\infty)}\leq1\\f\geq0}}\sup_{\substack{\|g\|_{Y'(0,\infty)}\leq1\\g\geq0}}\int_0^\infty f(s) s^\frac{m}{n}g^{**}(s)\,\d s\\
&= \sup_{\substack{\|g\|_{Y'(0,\infty)}\leq1\\g\geq0}}\|s^\frac{m}{n}g^{**}(s)\|_{X'(0,\infty)},
\end{align*}
where the last but one equality is true due to the Hardy-Littlewood inequality \eqref{E:HL} and the fact that $g$ and $g^*$ are equimeasurable.
\end{proof}

The following proposition provides a necessary condition on a pair $X$ and $Y$ of rearrangement-invariant spaces for the validity of \eqref{thm:eq:equivalentformsofonedimensionalinequality:reduction} or, equivalently, of~\eqref{thm:eq:equivalentformsofonedimensionalinequality:reductiondual}. This information will enable us to easily single out pairs of spaces for which \eqref{intr:homoemb} cannot hold after we have proved Theorem~\ref{thm:reductionprinciple}. Similar necessary conditions (sometimes called ``of Muckenhoupt type'' in the literature) have been treated in various contexts before and proved very useful, see e.g.~\cite[Theorem~1]{MR1320016} or~\cite[Lemma~1]{MR1267713}.

\begin{proposition}\label{prop:necessaryconforembedd}
Assume that $m < n$ and assume that $X(0,\infty)$ and $Y(0,\infty)$ are rearrangement-invariant spaces over $(0,\infty)$ such that \eqref{thm:eq:equivalentformsofonedimensionalinequality:reduction}, equivalently \eqref{thm:eq:equivalentformsofonedimensionalinequality:reductiondual}, is valid for them. Then
\begin{equation*}
\sup\limits_{0<a<\infty}\varphi_{Y(0,\infty)}(a)\|t^{\frac{m}{n} - 1}\chi_{(a,\infty)}(t)\|_{X'(0,\infty)} < \infty.
\end{equation*}
In particular,
\begin{equation*}
\|t^{\frac{m}{n} - 1}\chi_{(a,\infty)}(t)\|_{X'(0,\infty)} < \infty\quad\text{for each $a>0$}.
\end{equation*}
\end{proposition}
\begin{proof}
For each $a>0$ we have that
\begin{align*}
\|\chi_{(0,a)}\|_{Y(0,\infty)}\|t^{\frac{m}{n} - 1}\chi_{(a,\infty)}(t)\|_{X'(0,\infty)} &= \|\chi_{(0,a)}\|_{Y(0,\infty)}\sup\limits_{\|f\|_{X(0,\infty)}\leq1}\int_a^\infty |f(s)|s^{\frac{m}{n} - 1}\,\d s\\
&\leq\sup\limits_{\|f\|_{X(0,\infty)}\leq1}\left\|\chi_{(0,a)}(t)\int_t^\infty |f(s)|s^{\frac{m}{n} - 1}\,\d s\right\|_{Y(0,\infty)}\\
&\leq C_2,
\end{align*}
where $C_2$ is the constant from~\eqref{thm:eq:equivalentformsofonedimensionalinequality:reduction} or~\eqref{thm:eq:equivalentformsofonedimensionalinequality:reductiondual}.
\end{proof}

The following proposition is a key step in establishing the iteration principle of Theorem~\ref{thm:iterationprinciple}, which will also be indispensable in the proof of Theorem~\ref{thm:optimalrange}.
\begin{proposition}\label{prop:iterationprinciple}
Let $X(0,\infty)$ be a rearrangement-invariant space over $(0,\infty)$. Assume that $\alpha, \beta\in(0,\infty)$ are such that $\alpha + \beta < n$. Then there exist positive constants $C_1$ and $C_2$, depending on $\alpha$, $\beta$, and $n$ only, such that
\begin{align*}
C_1\|t^\frac{\alpha}{n}[\tau^\frac{\beta}{n}f^{**}(\tau)]^{**}(t)\|_{X(0,\infty)}&\leq\|t^\frac{\alpha+\beta}{n}f^{**}(t)\|_{X(0,\infty)}\\
&\leq C_2\|t^\frac{\alpha}{n}[\tau^\frac{\beta}{n}f^{**}(\tau)]^{**}(t)\|_{X(0,\infty)}\quad\text{for each}\ f\in\M(0,\infty).
\end{align*}
\end{proposition}
\begin{proof}
The first inequality was proved in \citep[Theorem~3.4]{CPnew} for $(0,1)$ instead of $(0,\infty)$. However, the proof works just as well for $(0,\infty)$ when combined with the argument from the proof of \citep[Lemma~4.10]{EMMP}. For the sake of brevity, the details are omitted.

Regarding the second inequality, we estimate
\begin{align*}
\|t^\frac{\alpha+\beta}{n}f^{**}(t)\|_{X(0,\infty)} &= \|t^{\frac{\alpha+\beta}{n}-1}\int_0^t f^{*}(s)\,\d s\|_{X(0,\infty)}\approx\|t^{\frac{\alpha}{n}-1}\int_t^{2t}\tau^{\frac{\beta}{n} - 1}\,\d\tau\int_0^t f^{*}(s)\,\d s\|_{X(0,\infty)}\\
&\leq\|t^{\frac{\alpha}{n}-1}\int_t^{2t}\tau^{\frac{\beta}{n} - 1}\int_0^\tau f^{*}(s)\,\d s\,\d\tau\|_{X(0,\infty)} = \|t^{\frac{\alpha}{n}-1}\int_t^{2t}\tau^\frac{\beta}{n}f^{**}(\tau)\,\d\tau\|_{X(0,\infty)}\\
&\leq\|t^\frac{\alpha}{n}[\tau^\frac{\beta}{n}f^{**}(\tau)]^{**}(t)\|_{X(0,\infty)},
\end{align*}
where Hardy-Littlewood inequality \eqref{E:HL} is exploited in the last step.
\end{proof}

Now we are in the position to prove our main results.
\begin{proof}[Proof of Theorem~\ref{thm:iterationprinciple}]
We have that
\begin{align*}
\sigma_{k}\left(t^{\frac{l}{n} - 1}\chi_{(1,\infty)}(t)\right) &= \|t^{\frac{k}{n} - 1}\int_0^t (s+1)^{\frac{l}{n} - 1}\,\d s\|_{X'(0,\infty)}\approx \|t^{\frac{k}{n} - 1}[(t+1)^\frac{l}{n} - 1]\|_{X'(0,\infty)}\\
&\leq \|t^{\frac{k}{n} - 1}[(t+1)^\frac{l}{n} - 1]\chi_{(0,1)}(t)\|_{X'(0,\infty)} + \|t^{\frac{k+l}{n} - 1}\chi_{(1,\infty)}(t)\|_{X'(0,\infty)}\\
&\leq\|t^{\frac{k}{n} - 1}[(t+1)^\frac{l}{n} - 1]\|_{L^\infty(0,1)}\|\chi_{(0,1)}\|_{X'(0,\infty)} + \|t^{\frac{k+l}{n} - 1}\chi_{(1,\infty)}(t)\|_{X'(0,\infty)}\\
&< \infty.
\end{align*}
Hence $t^{\frac{l}{n} - 1}\chi_{(1,\infty)}(t)\in (X^k)'(0,\infty)$. It follows from Proposition~\ref{prop:iterationprinciple} that
\begin{equation*}
\|u\|_{(X^k)^l}\approx\|u\|_{X^{k+l}},
\end{equation*}
where the multiplicative constants depend on $m$ and on the dimension $n$ only.
\end{proof}

\begin{proof}[Proof of~Theorem~\ref{thm:optimalrange}]
It can be proved that $\sigma_m$ is a rearrangement-invariant norm if and only if the condition \eqref{thm:optimalrange:cond} is satisfied (cf.~\citep[Theorem~5.4]{CPS} and \citep[Theorem~4.4]{EMMP}). We note only that the triangle inequality follows from \eqref{E:subadditivity-of-doublestar-a}. Observe that $t^{\frac{j}{n} - 1}\chi_{(1,\infty)}(t)\leq t^{\frac{m}{n} - 1}\chi_{(1,\infty)}(t)$ on $(0,\infty)$ for $j\in\{1,\dots,m\}$. Hence $\sigma_j$ is a~rearrangement-invariant norm too provided that $t^{\frac{m}{n} - 1}\chi_{(1,\infty)}(t)\in X'(0,\infty)$.

We shall prove \eqref{thm:optimalrange:vnoreni} by induction on $m$. Firstly, assume that $m=1$. Then \eqref{thm:eq:equivalentformsofonedimensionalinequality:reduction} with $m=1$, $Y=X^1$ and $C_2=1$ is true by Proposition~\ref{prop:equivalentformsofonedimensionalinequality}. Let $u\in V_0^1 X(\rn)$. Note that $\lim\limits_{t\rightarrow\infty}u^*(t) = 0$. Since $u^*$ is locally absolutely continuous (\citep[Lemma~4.1]{CP-ARK}), we can estimate
\begin{align*}
\|u\|_{X^1} &= \|u^*\|_{X^1(0,\infty)} = \left\|\int_t^\infty -\frac{\d u^*}{\d s}(s)\,\d s\right\|_{X^1(0,\infty)} = \left\|\int_t^\infty \left(-\frac{\d u^*}{\d s}(s)s^{1-\frac{1}{n}}\right)s^{\frac{1}{n} - 1}\,\d s\right\|_{X^1(0,\infty)}\\
&\leq\left\|-\frac{\d u^*}{\d s}(s)s^{1-\frac{1}{n}}\right\|_{X(0,\infty)}\lesssim\|\nabla u\|_{X},
\end{align*}
where the last inequality is valid with a multiplicative constant depending on the dimension $n$ only due to a generalized P\'{o}lya-Szeg\H{o} principle \citep[Lemma~4.1]{CP-ARK}.

Next, assume that $1<m<n$ and that we have already proved \eqref{thm:optimalrange:vnoreni} for all smaller values of $m$. Let $u\in V_0^m X(\rn)$. For each $i\in\{1,\dots, n\}$ we have that $\frac{\partial u}{\partial x_i}\in V_0^{m-1}X(\rn)$ and, by the induction hypothesis,
\begin{equation*}
\|\frac{\partial u}{\partial x_i}\|_{X^{m-1}}\lesssim\|\nabla^{m-1}\frac{\partial u}{\partial x_i}\|_X\lesssim\|\nabla^m u\|_X.
\end{equation*}
Hence
\begin{equation}\label{thm:optimalrange:eq1}
\|\nabla u\|_{X^{m-1}}\lesssim\|\nabla^m u\|_X,
\end{equation}
that is, $u\in V_0^1 X^{m-1}(\rn)$. By~Theorem~\ref{thm:iterationprinciple} we have that $t^{\frac{1}{n} - 1}\chi_{(1,\infty)}(t)\in (X^{m-1})'(0,\infty)$. Hence we are entitled to use the first step with $m=1$ for $X^{m-1}$ instead of $X$, which yields
\begin{equation}\label{thm:optimalrange:eq2}
\|u\|_{(X^{m-1})^1}\lesssim \|\nabla u\|_{X^{m-1}}.
\end{equation}
Using Theorem~\ref{thm:iterationprinciple} again it follows that
\begin{equation}\label{thm:optimalrange:eq3}
\|u\|_{(X^{m-1})^1}\approx\|u\|_{X^m},
\end{equation}
where the multiplicative constants depend on $m$ and on the dimension $n$ only.
Combining \eqref{thm:optimalrange:eq1}, \eqref{thm:optimalrange:eq2} and \eqref{thm:optimalrange:eq3}, we obtain the desired inequality \eqref{thm:optimalrange:vnoreni}.

We shall prove the optimality of $X^m$ now. Assume that
\begin{equation}\label{thm:optimalrange:eq4}
\|u\|_Y\lesssim\|\nabla^m u\|_X\quad\text{for each}\ u\in V_0^m X(\rn)
\end{equation}
for a rearrangement-invariant space $Y$ over $\rn$. We shall show that \eqref{thm:optimalrange:eq4} implies \eqref{thm:eq:equivalentformsofonedimensionalinequality:reduction}. The proof proceeds along the lines of the proof of \citep[Theorem~3.3]{ACPS}.
Let $f\in\Mpl(0,\infty)$ having a bounded support be given. We may assume that $\|f\|_{X(0,\infty)}<\infty$ because otherwise there is nothing to prove. Define a function $g$ by
\begin{equation*}
g(t) = \int_{\omega_n t^n}^\infty \int_{s_1}^\infty\cdots\int_{s_{m-1}}^\infty f(s_m)s_m^{\frac{m}{n} - m}\,\d s_m\cdots\d s_1,\ t\in(0,\infty).
\end{equation*}
Routine, albeit slightly tedious, computations show (cf.~\citep[(4.34) and (4.35)]{ACPS}) that for $k\in\{1,\dots, m-1\}$
\begin{equation}\label{thm:optimalrange:eq5}
|g^{(k)}(t)|\lesssim\sum_{l=1}^k t^{ln-k}\int_{\omega_n t^n}^\infty f(s)s^{\frac{m}{n}-l-1}\,\d s\quad\text{for each}\ t\in(0,\infty)
\end{equation}
and that
\begin{equation}\label{thm:optimalrange:eq6}
|g^{(m)}(t)|\lesssim f(\omega_n t^n) + \sum_{l=1}^{m-1} t^{ln-m}\int_{\omega_n t^n}^\infty f(s)s^{\frac{m}{n}-l-1}\,\d s\quad\text{for a.e.}\ t\in(0,\infty).
\end{equation}
Now, consider a function $u$ defined by
\begin{equation*}
u(x) = g(|x|),\ x\in\rn.
\end{equation*}
Then u is $m$-times weakly differentiable on $\rn$ and one can observe that
\begin{equation}\label{thm:optimalrange:eq7}
|\frac{\partial^m u}{\partial^{\alpha_1} x_1\cdots\partial^{\alpha_n} x_n}(x)|\lesssim\sum_{k=1}^m|g^{(k)}(|x|)||x|^{k-m}\quad\text{for a.e.}\ x\in\rn,
\end{equation}
where $\alpha_1+\cdots+\alpha_n = m$.
Hence, combining \eqref{thm:optimalrange:eq5} and \eqref{thm:optimalrange:eq6} with \eqref{thm:optimalrange:eq7}, we obtain that
\begin{equation}\label{thm:optimalrange:eq8}
|\nabla^m u(x)|\lesssim f(\omega_n |x|^n) + \sum_{l=1}^{m-1} |x|^{ln-m}\int_{\omega_n |x|^n}^\infty f(s)s^{\frac{m}{n}-l-1}\,\d s\quad\text{for a.e.}\ x\in\rn.
\end{equation}
Since for $l\in\{1,\dots, m-1\}$ the linear operator
\begin{equation}\label{thm:optimalrange:eq11}
f\mapsto\left(t\mapsto t^{l-\frac{m}{n}}\int_{t}^\infty f(s)s^{\frac{m}{n}-l-1}\,\d s\right)
\end{equation}
is bounded on both $L^1(0,\infty)$ and $L^\infty(0,\infty)$ and the corresponding operator norms depend on $l$ and on the dimension $n$ only, it is bounded on every rearrangement-invariant space over $(0,\infty)$ by \citep[Chapter~3, Theorem~2.2]{BS}. In particular, it is bounded on $X(0,\infty)$. Moreover, the operator norm of the operator \eqref{thm:optimalrange:eq11} on $X(0,\infty)$ can be bounded from above by a constant, which depends on $m$ and on the dimension $n$ only. Hence, using \eqref{thm:optimalrange:eq8}, we can estimate that
\begin{equation}\label{thm:optimalrange:eq9}
\|\nabla^m u\|_{X}\lesssim\|f\|_{X(0,\infty)} + \sum_{l=1}^{m-1}\|t^{l-\frac{m}{n}}\int_{t}^\infty f(s)s^{\frac{m}{n}-l-1}\,\d s\|_{X(0,\infty)}\lesssim \|f\|_{X(0,\infty)},
\end{equation}
where the multiplicative constants depend on $m$ and on the dimension $n$ only. Hence, $u\in V^m X(\rn)$. Furthermore, since $f$ has a bounded support, it follows that $u\in V_0^m X(\rn)$. By Fubini's theorem
\begin{equation*}
u(x) = \frac1{(m-1)!}\int_{\omega_n |x|^n}^\infty f(s)s^{\frac{m}{n} - m}(s-\omega_n|x|^n)^{m-1}\,\d s\quad\text{for }x\in\rn,
\end{equation*}
whence
\begin{equation}\label{thm:optimalrange:eq10}
\begin{split}
\|u\|_Y&\gtrsim \|\int_{2t}^\infty f(s)s^{\frac{m}{n} - m}(s-t)^{m-1}\,\d s\|_{Y(0,\infty)}\\
&\gtrsim\|\int_{2t}^\infty f(s)s^{\frac{m}{n} - m}s^{m-1}\,\d s\|_{Y(0,\infty)} = \|\int_{2t}^\infty f(s)s^{\frac{m}{n} - 1}\,\d s\|_{Y(0,\infty)},
\end{split}
\end{equation}
where the second inequality follows from the simple fact that $-t\geq-\frac{s}{2}$ for $s\geq 2t$.

Now, we are ready to finally establish \eqref{thm:eq:equivalentformsofonedimensionalinequality:reduction}. Indeed, by virtue of the boundedness of the dilation operator on rearrangement-invariant spaces, \eqref{thm:optimalrange:eq4}, \eqref{thm:optimalrange:eq9} and \eqref{thm:optimalrange:eq10}, we obtain that
\begin{align*}
\|\int_{t}^\infty f(s)s^{\frac{m}{n} - 1}\,\d s\|_{Y(0,\infty)}&\lesssim\|\int_{2t}^\infty f(s)s^{\frac{m}{n} - 1}\,\d s\|_{Y(0,\infty)}\lesssim \|u\|_Y\\
&\lesssim\|\nabla^m u\|_{X}\lesssim\|f\|_{X(0,\infty)}.
\end{align*}
Since an arbitrary function $f\in\Mpl(0,\infty)$ can be approximated by a nondecreasing sequence of nonnegative functions with bounded supports, \eqref{thm:eq:equivalentformsofonedimensionalinequality:reduction} follows. Since \eqref{thm:eq:equivalentformsofonedimensionalinequality:reduction} is equivalent to \eqref{thm:eq:equivalentformsofonedimensionalinequality:reductiondual} by Proposition~\ref{prop:equivalentformsofonedimensionalinequality}, we have, in fact, proved that $Y'\hookrightarrow (X^m)'$, equivalently, $X^m\hookrightarrow Y$.

Finally, if there exists any rearrangement-invariant space $Y$ over $\rn$ which renders \eqref{thm:optimalrange:eq4} true, then \eqref{thm:eq:equivalentformsofonedimensionalinequality:reduction} is valid by the computations above. Hence \eqref{thm:optimalrange:cond} is true by Proposition~\ref{prop:necessaryconforembedd}.
\end{proof}

\begin{proof}[Proof of~Theorem~\ref{thm:reductionprinciple}]
On the one hand, if \eqref{thm:eq:reductionprinciple:embedding} is valid, then $X^m\hookrightarrow Y$ by Theorem~\ref{thm:optimalrange}. Hence \eqref{thm:eq:equivalentformsofonedimensionalinequality:reductiondual} is valid by the very definition of $\sigma_m$, given by \eqref{thm:optimalrange:sigma}. On the other hand, assume that \eqref{thm:eq:equivalentformsofonedimensionalinequality:reductiondual} is in force, that is,
\begin{equation*}
\sigma_m(g)\lesssim\|g\|_{Y'(0,\infty)},
\end{equation*}
where $\sigma_m$ is defined by \eqref{thm:optimalrange:sigma}. Then \eqref{thm:optimalrange:cond} is satisfied by Proposition~\ref{prop:necessaryconforembedd} and
\begin{equation*}
X^m\hookrightarrow Y.
\end{equation*}
Hence by Theorem~\ref{thm:optimalrange}
\begin{equation*}
\|u\|_Y\lesssim\|u\|_{X^m}\lesssim\|\nabla^m u\|_{X}\quad\text{for each}\ u\in V_0^m X(\rn).
\end{equation*}
Thus the equivalence of \eqref{thm:eq:reductionprinciple:embedding} and \eqref{thm:eq:equivalentformsofonedimensionalinequality:reductiondual} has been proved. The inequalities \eqref{thm:eq:equivalentformsofonedimensionalinequality:reduction} and \eqref{thm:eq:equivalentformsofonedimensionalinequality:reductiondual} are equivalent by Proposition~\ref{prop:equivalentformsofonedimensionalinequality}.
\end{proof}

\begin{proof}[Proof of~Theorem~\ref{thm:optimaldomain}]
The fact that $\tau_m$ is a rearrangement-invariant norm is rather deep, especially the triangle inequality, and we refer the reader to \citep[Theorem~4.1]{EMMP}. Let $f\in\Mpl(0,\infty)$. Then
\begin{equation*}
\left\|\int_t^{\infty}f(s)s^{\frac{m}{n}-1}\,\d s\right\|_{Y(0,\infty)}\leq\sup_{\substack{h\sim f\\h\ge0}}
		\left\|\int_t^{\infty}h(s)s^{\frac{m}{n}-1}\,\d s\right\|_{Y(0,\infty)} = \|f\|_{Y_m(0,\infty)},
\end{equation*}
which proves \eqref{thm:optimaldomain:vnoreni} by Theorem~\ref{thm:reductionprinciple}.

Now, let $Z$ be a rearrangement-invariant space such that
\begin{equation*}
\|u\|_{Y}\leq C \|\nabla^m u\|_{Z}\quad\text{for each $u\in V_0^m Z(\rn)$}
\end{equation*}
and let $f\in\M(\rn)$ and $h\in\Mpl(0,\infty)$ be equimeasurable. We have that
\begin{equation*}
\left\|\int_t^{\infty}h(s)s^{\frac{m}{n}-1}\,\d s\right\|_{Y(0,\infty)}\lesssim \|h\|_{Z(0,\infty)} = \|f\|_{Z}
\end{equation*}
due to Theorem~\ref{thm:reductionprinciple}, whence
\begin{equation*}
\|f\|_{Y_m}\lesssim\|f\|_{Z}.
\end{equation*}
Hence $Z\hookrightarrow Y_m$.

Finally, if~\eqref{thm:optimaldomain:cond} is not true, then repeating the computations from the proof of \citep[Theorem~4.1]{EMMP}, one can prove that there is no rearrangement-invariant space $X$ for which~\eqref{thm:optimaldomain:vnoreni} is rendered true.
\end{proof}

\section{Examples of optimal Sobolev embeddings}\label{sec:riexamples}
In this section examples of optimal rearrangement-invariant spaces for Lorentz--Zygmund spaces and Orlicz spaces are given.
\begin{theorem}\label{thm:optimalrangeexamples}
Let $m < n$ and let $X=L^{p, q; \A}(\rn)$ where $p,q\in[1,\infty]$ and $\A\in\R^2$. Assume that one of the conditions \eqref{thm:optimalrangeexamples-condonglztobebfs} holds. The space $X^m$ defined by
\begin{equation}\label{thm:optimalrangeexamples:conds}
X^m =\begin{cases}L^{\frac{np}{n-mp}, q;\A},\quad&p=q=1,\ \alpha_0\geq0,\ \alpha_\infty\leq0\ \text{or}\\
&p\in(1,\frac{n}{m}),\\
L^{\infty, q;\A - 1},\quad&p=\frac{n}{m},\ \alpha_0 < \frac1{q'},\ \alpha_\infty>\frac1{q'},\\
L^{\infty, 1;[-1,\alpha_\infty - 1], [-1,0], [-1,0]},\quad&p=\frac{n}{m},\ q = 1,\ \alpha_0 = 0,\ \alpha_\infty > 0,\\
Y_1,\quad&p=\frac{n}{m},\ q = 1,\ \alpha_0 < 0 ,\ \alpha_\infty = 0,\\
L^\infty,\quad&p=\frac{n}{m},\ q = 1,\ \alpha_0 \geq 0 ,\ \alpha_\infty = 0,\\
Y_2,\quad&p=\frac{n}{m},\ q\in[1,\infty),\ \alpha_0 > \frac1{q'} ,\ \alpha_\infty > \frac1{q'},\\
L^{\infty, q;[-\frac1{q},\alpha_\infty - 1],[-1,0]},\quad&p=\frac{n}{m},\ q\in(1,\infty],\ \alpha_0 = \frac1{q'},\ \alpha_\infty >\frac1{q'},\\
L^{\infty, \infty;[0,\alpha_\infty - 1]},\quad&p=\frac{n}{m},\ q = \infty,\ \alpha_0 > 1,\ \alpha_\infty > 1,\end{cases}
\end{equation}
where
\begin{align*}
\|f\|_{Y_1} &=\|t^{-1}\ell^{\alpha_0 - 1}(t)f^*(t)\|_{L^1(0,1)},\\
\|f\|_{Y_2} &= \|f\|_{L^\infty} + \|t^{-\frac1{q}}\ell^{\alpha_\infty - 1}(t)f^*(t)\|_{L^q(1,\infty)},
\end{align*}
is the optimal (the smallest) target space for $X$ in \eqref{intr:homoemb}.

Conversely, if $p=\frac{n}{m}$ and $q=1$ and $\alpha_\infty<0$, or $p=\frac{n}{m}$ and $q\in(1,\infty]$ and $\alpha_\infty\leq\frac1{q'}$, or $p\in(\frac{n}{m},\infty]$, then there does not exist any rearrangement-invariant space $Y$ for which~\eqref{intr:homoemb} is true at all.
\end{theorem}

It turns out that the optimal target space for an Orlicz space $L^A$ depends on whether the integral
\begin{equation}\label{thm:optimalrangerforOrlicz:connek}
\int^\infty \left(\frac{s}{A(s)}\right)^\frac{m}{n-m}\,\d s
\end{equation}
converges or not. Assume that $m<n$ and that $A$ is a Young function such that
\begin{equation}\label{thm:optimalrangerforOrlicz:con0}
\int_0 \left(\frac{s}{A(s)}\right)^\frac{m}{n-m}\,\d s < \infty.
\end{equation}
Let $a$ be the left-continuous derivative of $A$, that is, $a$ and $A$ are related as in \eqref{pre:integralformofYoungfunction}.
We define a function $E_m$ by
\begin{equation}\label{thm:optimalrangerforOrlicz:defEm}
E_m(t) = \int _0^t e_m(s)\,\d s,\ t\geq0,
\end{equation}
where $e_m$ is the nondecreasing, left-continuous function in $[0, \infty)$ satysfying
\begin{equation*}
 e_m^{-1}(t) =  \left(\int _{a^{-1}(t)}^{\infty}\left(\int _0^s \left(\frac{1}{a(\tau)}\right)^{\frac{m}{n-m}}\,\d\tau\right)^{-\frac{n}{m}}\frac1{a(s)^{\frac{n}{n-m}}}\,\d s\right)^{\frac{m}{m-n}}\quad\text{for $t\geq0$}.
\end{equation*}
Then $E_m$ is a finite-valued Young function satisfying \eqref{pre:conditiononAforOrliczLorentz} with $p = \frac{n}{m}$ (see~\citep[Proposition~2.2]{MR2073127}).
\begin{theorem}\label{thm:optimalrangerforOrlicz}
Assume that $m < n$ and let $A$ be a Young function satisfying \eqref{thm:optimalrangerforOrlicz:con0}. Set
\begin{equation*}
X^m=\begin{cases}
L(\frac{n}{m}, 1, E_m),\quad&\text{the integral \eqref{thm:optimalrangerforOrlicz:connek} diverges},\\
L(\frac{n}{m}, 1, E_m)\cap L^\infty,\quad&\text{the integral \eqref{thm:optimalrangerforOrlicz:connek} converges},
\end{cases}
\end{equation*}
where $E_m$ is defined by \eqref{thm:optimalrangerforOrlicz:defEm}.

Then $X^m$ is the optimal (the smallest) target space for $L^A$ in \eqref{intr:homoemb}.

Conversely, if A does not satisfy~\eqref{thm:optimalrangerforOrlicz:con0}, then there does not exist any rearrangement-invariant space $Y$ for which~\eqref{intr:homoemb} is true with $X=L^A$ at all.
\end{theorem}
We also provide optimal domain spaces for Lorentz-Zygmund spaces.
\begin{theorem}\label{thm:optimaldomainexamples}
Let $m < n$ and let $Y=L^{p, q; \A}(\rn)$ where $p,q\in[1,\infty]$ and $\A\in\R^2$. Assume that one of the conditions \eqref{thm:optimalrangeexamples-condonglztobebfs} holds. The space $Y_m$ defined by
\begin{equation*}
Y_m =\begin{cases}L^{1,1;\A},\quad&p=\frac{n}{n-m},\ q=1,\ \alpha_0\geq0,\ \alpha_\infty\leq0,\\
X_1,\quad&p=\frac{n}{n-m},\ q=1,\ \alpha_0 < 0,\ \alpha_\infty\leq0\ \text{or}\\
					&p=\frac{n}{n-m},\ q\in(1,\infty],\ \alpha_\infty\leq0,\\
L^{\frac{np}{n + mp}, q;\A},\quad&p\in(\frac{n}{n-m},\infty),\\
X_2,\quad&p=\infty,\ q\in[1, \infty),\ \alpha_0 + \frac1{q} < 0\ \text{or}\\
					&p=q=\infty,\ \alpha_0\leq0,
\end{cases}
\end{equation*}
where
\begin{align*}
\|f\|_{X_1} &=\sup_{\substack{h\sim f\\h\ge0}}\|t^{1-\frac{m}{n} - \frac1{q}}\ell^{\A}(t)\int_t^\infty h(s)s^{\frac{m}{n} - 1}\,\d s\|_{L^q},\\
\|f\|_{X_2} &\approx\|t^{-\frac1{q}}\ell^{\A}(t)\int_t^\infty f^*(s)s^{\frac{m}{n} - 1}\,\d s\|_{L^q},
\end{align*}
is the optimal (the largest) domain space for $Y$ in \eqref{intr:homoemb}.

In particular, if $\A=[0,0]$, then $X_1 = L^1$ and $X_2=L^{\frac{n}{m}, 1}$.

Conversely, if either $p=\frac{n}{n-m}$ and $\alpha_\infty > 0$ or $p\in[1,\frac{n}{n-m})$, then there does not exist any rearrangement-invariant space $X$ for which~\eqref{intr:homoemb} is true at all.
\end{theorem}

\begin{proof}[Proof of~Theorem~\ref{thm:optimalrangeexamples}]
Note that $X$ is equivalent to a rearrangement-invariant space by \citep[Theorem~7.1]{OP} under our assumptions on $p,q,\A$, which entitles us to use Theorem~\ref{thm:optimalrange}. The condition \eqref{thm:optimalrange:cond} is satisfied if and only if one of the conditions \eqref{thm:optimalrangeexamples:conds} is satisfied. We skip these straightforward computations here and merely note that the description of $X'$ is given by \citep[Theorem~6.2 and Theorem~6.6]{OP}.

Let us turn our attention to \eqref{thm:optimalrangeexamples:conds}. Using \eqref{thm:optimalrange:sigma} and \citep[Theorem~6.2 and Theorem~6.6]{OP}, we have that
\begin{align*}
\|f\|_{(X^m)'} &= \|t^{\frac1{p'}-\frac1{q'}}\ell^{-\A}(t)\left[\tau^\frac{m}{n}f^{**}(\tau)\right]^*(t)\|_{L^{q'}}\leq \|t^{\frac1{p'}-\frac1{q'}}\ell^{-\A}(t)\sup\limits_{t\leq\tau < \infty}\tau^\frac{m}{n}f^{**}(\tau)\|_{L^{q'}}\\
&\lesssim \|t^{\frac1{p'}-\frac1{q'} + \frac{m}{n}}\ell^{-\A}(t)f^{**}(t)\|_{L^{q'}} = \|f\|_{L^{(\frac{np'}{n+mp'},q'; -\A)}},
\end{align*}
where $\frac{np'}{n+mp'}$ is to be interpreted as $\frac{n}{m}$ if $p = 1$. The first inequality follows from the very definition of the nonincreasing rearrangement. The validity of the last inequality is due to \citep[Theorem~3.2]{GOP} if $q\in(1,\infty]$. If $q=1$, then its validity is due to the fact that
\begin{equation*}
\sup\limits_{t>0}t^\frac1{p'}\ell^{-\A}(t)\sup\limits_{t\leq\tau<\infty}\tau^\frac{m}{n}f^{**}(\tau)
= \sup\limits_{\tau>0}\tau^\frac{m}{n}f^{**}(\tau)\sup\limits_{0<t\leq\tau}t^\frac1{p'}\ell^{-\A}(t)\approx\sup\limits_{\tau>0}\tau^{\frac{m}{n} + \frac1{p'}}\ell^{-\A}(\tau)f^{**}(\tau)
\end{equation*}
since the function $t\mapsto t^\frac1{p'}\ell^{-\A}(t)$ is equivalent to a nondecreasing function on $(0, \infty)$ if $p > 1$, and if $p = 1$, then the function $t\mapsto t^\frac1{p'}\ell^{-\A}(t) = \ell^{-\A}(t)$ is nondecreasing on $(0, \infty)$ as $-\alpha_0 \leq0$ and $-\alpha_\infty\geq0$. On the other hand,
\begin{align*}
\|f\|_{L^{(\frac{np'}{n+mp'},q'; -\A)}} &= \|t^{\frac1{p'}-\frac1{q'} + \frac{m}{n}}\ell^{-\A}(t)f^{**}(t)\|_{L^{q'}} \leq
\|t^{\frac1{p'}-\frac1{q'}}\ell^{-\A}(t)\sup\limits_{t\leq\tau<0}\tau^\frac{m}{n}f^{**}(\tau)\|_{L^{q'}}\\
&= \|\sup\limits_{t\leq\tau < \infty}\tau^\frac{m}{n}f^{**}(\tau)\|_{L^{p',q';-\A}}\lesssim\|t^\frac{m}{n}f^{**}(t)\|_{L^{p',q';-\A}}\\
&= \|f\|_{(X^m)'},
\end{align*}
where the last inequality is true thanks to \citep[Lemma~4.10]{EMMP}.

Hence we have shown that $(X^m)'$ is equivalent to $L^{(\frac{np'}{n+mp'},q'; -\A)}$, that is, $X^m$ is equivalent to $\left(L^{(\frac{np'}{n+mp'},q'; -\A)}\right)'$. The assertion then follows from the description of the associate space of $L^{(\frac{np'}{n+mp'},q'; -\A)}$. If $p < \frac{n}{m}$, then $\frac{np'}{n+mp'} > 1$ and $L^{(\frac{np'}{n+mp'},q'; -\A)}$ is equivalent to $L^{\frac{np'}{n+mp'},q'; -\A}$ by \citep[Theorem~3.8]{OP} and its associate space is described by \citep[Theorem~6.2, Theorem~6.6]{OP}. If $p = \frac{n}{m}$, then $\frac{np'}{n+mp'} = 1$ and the associate space of $L^{(1,q'; -\A)}$ is given by \citep[Theorem~6.7, Theorem~6.9]{OP}.
\end{proof}

\begin{proof}[Proof of Theorem~\ref{thm:optimalrangerforOrlicz}]
Let $X=L^A$.  It follows from \citep[Theorem~3.1]{MR2073127} (cf.~also \citep[(3.1) and Remark~3.2]{MR2073127}) that
\begin{equation}\label{thm:optimalrangerforOrlicz:eq1}
\|t^{-\frac{m}{n}}\int_t^\infty s^{\frac{m}{n}-1}f(s)\,\d s\|_{L^{E_m}(0,\infty)}\lesssim\|f\|_{L^A(0,\infty)}\quad\text{for each $f\in L^A(0,\infty)$}.
\end{equation}
In particular, if $f\in L^A(0,\infty)$, then $\int_t^\infty s^{\frac{m}{n}-1}f(s)\,\d s\in L(\frac{n}{m}, 1, E_m)(0,\infty)$. Hence
\begin{align*}
\|g\|_{\left(L(\frac{n}{m}, 1, E_m)\right)'} &= \sup_{\substack{f\in L(\frac{n}{m}, 1, E_m)\\f\neq0}}\frac{\int_0^\infty f^*(t)g^*(t)\,\d t}{\|f\|_{L(\frac{n}{m}, 1, E_m)}}\\
&\geq \sup_{\substack{f\in L^A(0,\infty)\\f\neq0}}\frac{\int_0^\infty g^*(t)\int_t^\infty s^{\frac{m}{n}-1}|f(s)|\,\d s\,\d t}{\|\int_t^\infty s^{\frac{m}{n}-1}|f(s)|\,\d s\|_{L(\frac{n}{m}, 1, E_m)(0,\infty)}}\\
&= \sup_{\substack{f\in L^A(0,\infty)\\f\neq0}}\frac{\int_0^\infty |f(s)|s^{\frac{m}{n}-1}\int_0^s g^*(t)\,\d t\,\d s}{\|t^{-\frac{m}{n}}\int_t^\infty s^{\frac{m}{n}-1}|f(s)|\,\d s\|_{L^{E_m}(0,\infty)}}\\
&\gtrsim\sup_{\substack{f\in L^A(0,\infty)\\f\neq0}}\frac{\int_0^\infty |f(s)|s^{\frac{m}{n}-1}\int_0^s g^*(t)\,\d t\,\d s}{\|f\|_{L^A(0,\infty)}}\\
&= \|t^\frac{m}{n}g^{**}(t)\|_{(L^A)'(0,\infty)}=\|g\|_{(X^m)'},
\end{align*}
where the last inequality is true thanks to \eqref{thm:optimalrangerforOrlicz:eq1}. Hence \eqref{thm:eq:equivalentformsofonedimensionalinequality:reduction} holds with $Y(0,\infty) = L(\frac{n}{m}, 1, E_m)(0,\infty)$ by Proposition~\ref{prop:equivalentformsofonedimensionalinequality}.

If the integral \eqref{thm:optimalrangerforOrlicz:connek} diverges, we have that
\begin{align*}
\|g\|_{\left(L(\frac{n}{m}, 1, E_m)\right)'} &= \sup_{\substack{f\in L(\frac{n}{m}, 1, E_m)\\f\neq0}}\frac{\int_0^\infty f^*(t)g^*(t)\,\d t}{\|f\|_{L(\frac{n}{m}, 1, E_m)}}\\
&\lesssim\sup_{\substack{f\in L(\frac{n}{m}, 1, E_m)\\f\neq0}}\frac{\|t^{-\frac{m}{n}}f^*(t)\|_{L^{E_m}(0,\infty)}\|t^{\frac{m}{n}}g^{**}(t)\|_{(L^A)'(0,\infty)}}{\|f\|_{L(\frac{n}{m}, 1, E_m)}}\\
&= \|g\|_{(X^m)'},
\end{align*}
where the inequality is due to \citep[Theorem~4.1, (4.2)]{MR2073127}. Hence $X^m$ is equivalent to $L(\frac{n}{m}, 1, E_m)$ by virtue of the equivalence of \eqref{pre:embedding} and \eqref{pre:embeddingdual}.

Now, assume that the integral \eqref{thm:optimalrangerforOrlicz:connek} converges. Then
\begin{equation*}
\|\int_t^\infty f(s)s^{\frac{m}{n} - 1}\,\d s\|_{L^\infty(0,\infty)}\lesssim \|t^{\frac{m}{n} - 1}\|_{L^{\widetilde{A}}(0,\infty)}\|f\|_{L^A (0,\infty)} \approx \left(\int_0^\infty\frac{\widetilde{A}(s)}{s^{1+\frac{n}{n-m}}}\,\d s\right)^\frac{n-m}{n}\|f\|_{L^A (0,\infty)},
\end{equation*}
where the integral on the right-hand side is finite thanks to \citep[Lemma~2.3]{MR2073127}. This together with the estimate at the beginning of this proof ensures that \eqref{intr:homoemb} is true with $Y = L(\frac{n}{m}, 1, E_m)\cap L^\infty$ by virtue of Theorem~\ref{thm:reductionprinciple}. The optimality can be shown along the same lines of \citep[Theorem~1.1, pp.~457]{MR2073127} and we omit it here.

Finally, should $t^{\frac{m}{n} - 1}\chi_{(1,\infty)}(t)\in(L^A)'(0,\infty)$ for a Young function $A$, then \eqref{thm:optimalrangerforOrlicz:con0} is necessarily satisfied. This can be proved along the lines of~\citep[Corollary~2.1]{MR2073127}. Hence if \eqref{thm:optimalrangerforOrlicz:con0} is not true, then there is no target space for $L^A$ in \eqref{intr:homoemb} by Theorem~\ref{thm:optimalrange}.
\end{proof}

By Theorem~\ref{thm:simplificationofdomainnorm} the description of the optimal domain space for $Y$ can be significantly simplified provided that the operator $T_\frac{m}{n}$, defined by \eqref{rem:simplificationofdomainnormsupopdef}, is bounded on the representation space of $Y'$. For this reason, it is convenient to know when the operator is bounded on the associate spaces of Lorentz-Zygmund spaces.
\begin{proposition}\label{prop:boundednessofTalphaonglz}
Let $X(0,\infty) = L^{p,q;\A}(0,\infty)$ and assume that one of the conditions \eqref{thm:optimalrangeexamples-condonglztobebfs} holds. Let $\alpha\in(0,1)$. Then $T_\alpha$ is bounded on $X'(0,\infty)$ if and only if
\begin{align*}
\text{either}\ p &= \frac1{1-\alpha},\ q = 1,\ \alpha_0\geq0\ \text{and}\ \alpha_\infty\leq0\\
\text{or}\ p&\in(\frac1{1-\alpha}, \infty].
\end{align*}
\end{proposition}
\begin{proof}
If $p = \frac1{1-\alpha}$, $q=1$, $\alpha_0\geq0$ and $\alpha_\infty\leq0$, or $p\in(\frac1{1-\alpha}, \infty)$, or $p=q=\infty$ and $\alpha_\infty\geq0$, then $T_\alpha$ is bounded on $X'(0,\infty)$. On the other hand, if $p\in[1, \frac1{1-\alpha})$, or $p = \frac1{1-\alpha}$, $q=1$, $\alpha_0 < 0$ or $\alpha_\infty > 0$, or $p = \frac1{1-\alpha}$ and $q\in(1,\infty]$, then $T_\alpha$ is not bounded on $X'(0,\infty)$. These facts follow from the fact that the associate space of $X(0,\infty)$ is $L^{p', q'; -\A}$ (cf.~\citep[Theorem~6.2, Theorem~6.6]{OP}) and the fact that $T_\alpha$ is bounded on $X(0,\infty)$ if and only if
\begin{align*}
\text{either}\ p&\in[1, \frac1{\alpha})\\
\text{or}\ p&=\frac1{\alpha},\ q=\infty,\ \alpha_0\leq0\ \text{and}\ \alpha_\infty\geq0,
\end{align*}
which was shown in the proof of \citep[Theorem~4.5]{EMMP}.

Now, we shall prove that $T_\alpha$ is bounded on $X'(0,\infty)$ in the remaining cases, that is, $p=\infty$ and $q\in[1,\infty)$, or $p=q=\infty$ and $\alpha_\infty < 0$. Assume that $p=q=\infty$, $\alpha_\infty < 0$. Then by \citep[Theorem~6.2]{OP} the norm on $X'(0,\infty)$ is given by
\begin{equation*}
\|f\|_{X'(0,\infty)} = \|\ell^{-\alpha_0}(t)f^*(t)\|_{L^1(0,1)} + \|f\|_{L^1(0,\infty)},
\end{equation*}
and $T_\alpha$ is bounded on $X'(0,\infty)$ because
\begin{align*}
\|T_\alpha f\|_{X'(0,\infty)} &= \int_0^1 \ell^{-\alpha_0}(t)\left[s^{-\alpha}\sup\limits_{s\leq\tau < \infty}\tau^\alpha f^*(\tau)\right]^*(t)\,\d t + \int_0^\infty\left[s^{-\alpha}\sup\limits_{s\leq\tau < \infty}\tau^\alpha f^*(\tau)\right]^*(t)\,\d t\\
&=\int_0^\infty\left(\ell^{-\alpha_0}(t)\chi_{(0,1)}(t) + 1\right)\left[s^{-\alpha}\sup\limits_{s\leq\tau < \infty}\tau^\alpha f^*(\tau)\right]^*(t)\\
&\leq \int_0^\infty\left(\ell^{-\alpha_0}(t)\chi_{(0,1)}(t) + 1\right)\sup\limits_{t\leq s < \infty}s^{-\alpha}\sup\limits_{s\leq\tau < \infty}\tau^\alpha f^*(\tau)\,\d t\\
&=\int_0^\infty\left(\ell^{-\alpha_0}(t)\chi_{(0,1)}(t) + 1\right)t^{-\alpha}\sup\limits_{t\leq\tau < \infty}\tau^\alpha f^*(\tau)\,\d t\\
&\lesssim\int_0^\infty\left(\ell^{-\alpha_0}(t)\chi_{(0,1)}(t) + 1\right)f^*(t)\,\d t\\
&=\|f\|_{X'(0,\infty)},
\end{align*}
where the last inequality is true due to \citep[Theorem~3.2]{GOP}.

If $p=\infty$, $q\in[1,\infty)$ and $\alpha_\infty + \frac1{q}\geq0$, then $X'$ is $L^{(1, q'; \B, \C)}$ for appropriate $\B, \C\in\R^2$ (cf.~\citep[Theorem~6.2, Theorem~6.6]{OP}). It follows from \citep[Lemma~4.1]{MusilOlhava2018} that
\begin{equation*}
(T_\alpha f)^{**}(t)\lesssim T_\alpha f^{**}(t)\quad\text{for each}\ f\in\M(0,\infty),\ t > 0.
\end{equation*}
Hence
\begin{align*}
\|T_\alpha f\|_{X'(0,\infty)} &= \|t^{1-\frac1{q'}}\ell^\B(t)\ell\ell^\C(t)(T_\alpha f)^{**}(t)\|_{q'}\lesssim\|t^{1-\frac1{q'}}\ell^\B(t)\ell\ell^\C(t)T_\alpha f^{**}(t)\|_{q'}\\
&\lesssim \|t^{1-\frac1{q'}}\ell^\B(t)\ell\ell^\C(t)f^{**}(t)\|_{q'} = \|f\|_{X'(0,\infty)},
\end{align*}
where the last inequality is true thanks to \citep[Theorem~3.2]{GOP} if $q\in(1,\infty)$. If $q=1$, then the last inequality is in fact an equality (up to a positive multiplicative constant), which follows from interchanging the order of the suprema and the fact that the function
\begin{equation*}
t\mapsto t\ell^\B(t)\ell\ell^\C(t)
\end{equation*}
is equivalent to a nondecreasing function on $(0,\infty)$.

Finally, if $p=\infty$, $q\in[1,\infty)$ and $\alpha_\infty + \frac1{q} < 0$, we can proceed similarly, omitting the proof here.
\end{proof}

\begin{proof}[Proof of Theorem~\ref{thm:optimaldomainexamples}]
Since a rearrangement-invariant space $X$ is the optimal (the largest) domain space for a given rearrangement-invariant space $Y$ in the inequality \eqref{intr:homoemb} if and only if $X'$ is the optimal (the smallest) range partner for $Y'$ with respect to $M_{m}$ (cf.~Remark~\ref{rem:relationshipbetweenembeddingandfractional}), the theorem follows from Theorem~\ref{thm:optimaldomain}, \citep[Theorem~4.5]{EMMP}, Theorem~\ref{thm:simplificationofdomainnorm} and the Proposition~\ref{prop:boundednessofTalphaonglz} with $\alpha=\frac{m}{n}$.
\end{proof}

\section{Optimal embeddings of Orlicz--Sobolev spaces into Orlicz spaces}\label{sec:orliczexamples}
By Theorem~\ref{thm:reductionprinciple} the question of optimality in \eqref{intr:homoemb} is equivalent to the question of optimality in the one-dimensional inequality \eqref{thm:eq:equivalentformsofonedimensionalinequality:reduction}. The latter question was extensively studied (among other things) within the class of Orlicz spaces in \citep[Chapter~3]{VejtekPhD}. This enables us to look for optimal spaces in \eqref{intr:homoemb} within the class of Orlicz spaces. Since the optimal Orlicz space (provided that it exists) for an Orlicz space is sometimes simpler to describe than the corresponding optimal rearrangement-invariant space, especially in limit cases, the optimal Orlicz space is sometimes more convenient for applications.

We say that an Orlicz space $L^B$ over $\rn$ is the \emph{optimal target space within the class of Orlicz spaces} for an Orlicz space $L^A$ over $\rn$ in \eqref{intr:homoembOrlicz} if \eqref{intr:homoembOrlicz} is satisfied and whenever \eqref{intr:homoembOrlicz} is satisfied for another Orlicz space $L^C$ over $\rn$ in place of $L^B$, $L^C$ is larger than $L^B$, that is, $L^B\hookrightarrow L^C$. We say that an Orlicz space $L^A$ over $\rn$ is the \emph{optimal domain space within the class of Orlicz spaces} for an Orlicz space $L^B$ over $\rn$ in \eqref{intr:homoembOrlicz} if \eqref{intr:homoembOrlicz} is satisfied and whenever \eqref{intr:homoembOrlicz} is satisfied for another Orlicz space $L^C$ over $\rn$ in place of $L^A$, $L^C$ is smaller than $L^A$, that is, $L^C\hookrightarrow L^A$. We stress that the key difference from the prior sections is that the competing spaces are from the class of Orlicz spaces only, not from the class of all rearrangement-invariant spaces.

As it was already noted in Remark~\ref{rem:relationshipbetweenembeddingandfractional}, there is an intimate connection between the inequality \eqref{intr:homoembOrlicz} and the boundedness of the fractional maximal operator.  Optimality of Orlicz spaces for the latter was studied in \citep{VejtekPhD, Musil2019}. The combination of these results with appropriate duality principles appears to be useful for our purposes. We omit proofs in this section because they are lengthy and technical. The interested reader can trace the key ideas in \citep{VejtekPhD, Musil2019}.

Let $m < n$ and let $A$ be a Young function satisfying \eqref{thm:optimalrangerforOrlicz:con0}. We set
\begin{equation*}\label{thm:optimalorliczrange:defHinf}
H^\infty =\lim\limits_{t\rightarrow\infty}H_m(t)
\end{equation*}
where $H_m$ is defined by
\begin{equation*}\label{thm:optimalorliczrange:defHm}
H_m(t) = \left(\int_0^t \left(\frac{s}{A(s)}\right)^\frac{m}{n-m}\,\d s\right)^\frac{n-m}{n},\ t\geq0.
\end{equation*}
Note that $H^\infty = \infty$ if and only if the integral \eqref{thm:optimalrangerforOrlicz:connek} diverges.
Finally, we define
\begin{equation}\label{thm:optimalorliczrange:defDm}
D_m(t)=\begin{cases}
\left(t\frac{A(H_m^{-1}(t))}{H_m^{-1}(t)}\right)^\frac{n}{n-m},\quad&0\leq t < H^\infty,\\
\infty,\quad&H^\infty\leq t <\infty.
\end{cases}
\end{equation}

The following theorem is an application of Theorem~\ref{thm:reductionprinciple} and \citep[Theorem~3.4.1]{VejtekPhD}.
\begin{theorem}\label{thm:optimalorliczrange}
Let $m < n$ and let $A$ be a Young function satisfying \eqref{thm:optimalrangerforOrlicz:con0}. Define the Young function $A_m$ by
\begin{equation}\label{thm:optimalorliczrange:defAm}
A_m(t) = \int_0^t\frac{D_m(s)}{s}\,\d s,\ t\geq0,
\end{equation}
where the function $D_m$ is defined by \eqref{thm:optimalorliczrange:defDm}.

Then the Orlicz space $L^{A_m}$ is the optimal (the smallest) target space for $L^A$ in \eqref{intr:homoembOrlicz} within the class of Orlicz spaces.

Conversely, if~\eqref{thm:optimalrangerforOrlicz:con0} is not true, then there does not exist any Orlicz space $L^B$ for which~\eqref{intr:homoembOrlicz} is true at all.
\end{theorem}

\begin{remark}
The condition \eqref{thm:optimalrangerforOrlicz:con0} is, in fact, also necessary for existence of a target space even in the wider class of rearrangement-invariant spaces (cf.~Theorem~\ref{thm:optimalrangerforOrlicz}).

It is worth noting that (see~\citep[(3.3.6)]{VejtekPhD}) $A_m$ is equivalent to $D_m$ globally. Moreover, either $A_m$ is equivalent to $A\circ H_m^{-1}$ globally if the integral \eqref{thm:optimalrangerforOrlicz:connek} diverges or $A_m$ is equivalent to $A\circ H_m^{-1}$ near zero and $A_m(t)=\infty$ near infinity if the integral \eqref{thm:optimalrangerforOrlicz:connek} converges (see~\citep[(3.3.10)]{VejtekPhD}).

If $I_A < \frac{n}{m}$, where $I_A$ is the upper Boyd index of $A$, defined by \eqref{thm:optimalorliczdomain:Boydupp}, then (see~\citep[(3.4.2)]{VejtekPhD})
\begin{equation*}
A_m^{-1}(t)\approx A^{-1}(t)t^{-\frac{m}{n}}\quad\text{for $t>0$}.
\end{equation*}
\end{remark}

By standard calculations, one can use Theorem~\ref{thm:optimalorliczrange} to obtain optimal Orlicz spaces for some customary Orlicz spaces.
\begin{theorem}
Let $p_0,\,p_\infty\in[1,\infty)$ and $\alpha_0,\,\alpha_\infty\in\R$. Assume that if $p_0=1$, then $\alpha_0\leq0$, and if $p_\infty=1$, then $\alpha_\infty\geq0$.  Let $A(t)$ be a Young function that is equivalent to
\begin{equation*}
\begin{cases}
t^{p_0}\ell^{\alpha_0}(t)\quad&\text{near zero},\\
t^{p_\infty}\ell^{\alpha_\infty}(t)\quad&\text{near infinity}.
\end{cases}
\end{equation*}
The Young function $A_m(t)$, defined by \eqref{thm:optimalorliczrange:defAm}, is equivalent to
\begin{equation*}
\begin{cases}
t^{\frac{np_0}{n-mp_0}}\ell^{\frac{n\alpha_0}{n-mp_0}}(t),\quad&\text{$p_0\in[1,\frac{n}{m})$},\\
e^{-t^{\frac{n}{n-(1+\alpha_0)m}}},\quad&\text{$p_0=\frac{n}{m}$, $\alpha_0>\frac{n-m}{m}$},
\end{cases}
\end{equation*}
near zero and to
\begin{equation*}
\begin{cases}
t^{\frac{np_\infty}{n-mp_\infty}}\ell^{\frac{n\alpha_\infty}{n-mp_\infty}}(t),\quad&\text{$p_\infty\in[1,\frac{n}{m})$},\\
e^{t^{\frac{n}{n-(1+\alpha_\infty)m}}},\quad&\text{$p_\infty=\frac{n}{m}$, $\alpha_\infty < \frac{n-m}{m}$},\\
e^{e^{t^{\frac{n}{n - m}}}},\quad&\text{$p_\infty=\frac{n}{m}$, $\alpha_\infty = \frac{n-m}{m}$},\\
\infty,\quad&\text{$p_\infty=\frac{n}{m}$, $\alpha_\infty>\frac{n-m}{m}$ or}\\
\quad&\text{$p_\infty\in(\frac{n}{m},\infty)$},
\end{cases}
\end{equation*}
near infinity and the Orlicz space $L^{A_m}$ is the optimal (the smallest) target space for $L^A$ in \eqref{intr:homoembOrlicz} within the class of Orlicz spaces.

Conversely, if either $p_0=\frac{n}{m}$ and $\alpha_0\leq\frac{n-m}{m}$ or $p_0\in(\frac{n}{m},\infty)$, then there does not exist any Orlicz space $L^B$ for which~\eqref{intr:homoembOrlicz} is true at all.
\end{theorem}

To complement Theorem~\ref{thm:optimalorliczrange}, we now address the question of optimal domain spaces within the class of Orlicz spaces. If $m < n$ and $B$ is a Young function satisfying
\begin{equation}\label{thm:optimalorliczdomain:condonB}
\sup\limits_{0<t<1}\frac{B(t)}{t^{\frac{n}{n-m}}}<\infty,
\end{equation}
we define the function $G_m$ by
\begin{equation}\label{thm:optimalorliczdomain:defGm}
G_m(t)=t\inf_{0<s\leq t}B^{-1}(s)s^\frac{m-n}{n},\ t > 0.
\end{equation}
It follows from \eqref{thm:optimalorliczdomain:condonB} that $G_m$ is a positive function on $(0,\infty)$.

The following theorem is an application of Theorem~\ref{thm:reductionprinciple} and \citep[Theorem~3.6.1]{VejtekPhD}.
\begin{theorem}\label{thm:optimalorliczdomain}
Let $m < n$ and let $B$ be a Young function satisfying \eqref{thm:optimalorliczdomain:condonB}. Define the Young function $B_m$ by
\begin{equation}\label{thm:optimalorliczdomain:defBm}
B_m(t)=\int_0^t\frac{G_m^{-1}(s)}{s}\,\d s,\ t\geq0,
\end{equation}
where the function $G_m$ is defined by \eqref{thm:optimalorliczdomain:defGm}.

If $I_{B_m} < \frac{n}{m}$, then the Orlicz space $L^{B_m}$ is the optimal (the largest) domain space for $L^B$ in \eqref{intr:homoembOrlicz} within the class of Orlicz spaces.

If $I_{B_m} \geq \frac{n}{m}$, then there is no optimal Orlicz domain space for $L^B$ in \eqref{intr:homoembOrlicz} in the sense that whenever $L^A$ is an Orlicz space that renders \eqref{intr:homoembOrlicz} true, there exists an Orlicz space $L^C$ such that $L^A\subsetneq L^C$ that still renders \eqref{intr:homoembOrlicz} with $L^C$ instead of $L^A$ true.

Conversely, if~\eqref{thm:optimalorliczdomain:condonB} is not true, then there does not exist any Orlicz space $L^A$ for which~\eqref{intr:homoembOrlicz} is true at all.
\end{theorem}

\begin{remark}
Assume that $Y=L^B$ is an Orlicz space. Note that the conditions \eqref{thm:optimalorliczdomain:condonB} and \eqref{thm:optimaldomain:cond} are equivalent. Hence not only is there no Orlicz space $L^A$ for which~\eqref{intr:homoembOrlicz} is true if \eqref{thm:optimalorliczdomain:condonB} is not satisfied, but there is no rearrangement-invariant space $X$ for which~\eqref{intr:homoemb} is true at all. We would also like to stress the significant difference between Theorem~\ref{thm:optimalorliczdomain} and Theorem~\ref{thm:optimaldomain}. Whereas there always exists the optimal rearrangement-invariant domain space for a given rearrangement-invariant space $Y$ in \eqref{intr:homoemb} if there exists any rearrangement-invariant domain space, the situation is more complicated within the class of Orlicz spaces. If a Young function $B$ satisfies \eqref{thm:optimalorliczdomain:condonB}, we can define the Young function $B_m$ by \eqref{thm:optimalorliczdomain:defBm}. If $I_{B_m} \geq \frac{n}{m}$, then \eqref{intr:homoembOrlicz} with $L^{B_m}$ on the right-hand side is not satisfied because the Orlicz space $L^{B_m}$ is ``too large''; however, there still exist some Orlicz spaces $L^A$ that render \eqref{intr:homoembOrlicz} true but none of them is optimal. In this situation, we have, loosely speaking, an open set of Orlicz spaces $L^A$ that renders \eqref{intr:homoembOrlicz} true.

It can be shown (see~\citep[(3.5.8)]{VejtekPhD}) that
\begin{equation*}
B_m^{-1}(t)\approx G_m(t)\quad\text{for $t>0$}.
\end{equation*}
Moreover, if $i_B > \frac{n}{n - m}$, where $i_B$ is the lower Boyd index of $B$, defined by \eqref{thm:optimalorliczdomain:Boydlow}, then (see~\citep[(3.6.3)]{VejtekPhD})
\begin{equation}
B_m^{-1}(t)\approx B^{-1}(t)t^\frac{m}{n}\quad\text{for $t > 0$}
\end{equation}
and $I_{B_m} < \frac{n}{m}$ is equivalent to $I_B < \infty$.
\end{remark}

\begin{theorem}
Let $p_0,\,p_\infty\in[1,\infty)$ and $\alpha_0,\,\alpha_\infty\in\R$. Assume that if $p_0=1$, then $\alpha_0\leq0$, and if $p_\infty=1$, then $\alpha_\infty\geq0$.  Let $B(t)$ be a Young function that is equivalent to
\begin{equation*}
\begin{cases}
t^{p_0}\ell^{\alpha_0}(t)\quad&\text{near zero},\\
t^{p_\infty}\ell^{\alpha_\infty}(t)\quad&\text{near infinity}.
\end{cases}
\end{equation*}
If either $p_0=\frac{n}{n-m}$ and $\alpha_0\leq0$ or $p_0\in(\frac{n}{n-m},\infty)$, then the Young function $B_m(t)$, defined by \eqref{thm:optimalorliczdomain:defBm}, is equivalent to
\begin{equation*}
t^{\frac{np_0}{n+mp_0}}\ell^{\frac{n\alpha_0}{n+mp_0}}(t)\quad\text{near zero}
\end{equation*}
and to
\begin{equation*}
\begin{cases}
t^{\frac{np_\infty}{n+mp_\infty}}\ell^{\frac{n\alpha_\infty}{n+mp_\infty}}(t),\quad&\text{$p_\infty\in(\frac{n}{n-m},\infty)$ or},\\
\quad&\text{$p_\infty=\frac{n}{n-m}$, $\alpha_\infty>0$},\\
t,\quad&\text{$p_\infty=\frac{n}{n-m}$, $\alpha_\infty\leq0$ or},\\
\quad&\text{$p_\infty\in[1,\frac{n}{n-m})$},
\end{cases}
\end{equation*}
near infinity and the Orlicz space $L^{B_m}$ is the optimal (the largest) domain space for $L^B$ in \eqref{intr:homoembOrlicz} within the class of Orlicz spaces.

Conversely, if either $p_0=\frac{n}{n - m}$ and $\alpha_0>0$ or $p_0\in[1,\frac{n}{n - m})$, then there does not exist any Orlicz space $L^A$ for which~\eqref{intr:homoembOrlicz} is true at all.
\end{theorem}

\begin{remark}
Loosely speaking, the optimal domain space for $L^B$  in \eqref{intr:homoembOrlicz} within the class of Orlicz spaces exists provided that the Orlicz space $L^B$ is ``far from $L^\infty$''. On the other hand, Orlicz domain spaces for Orlicz spaces ``near $L^\infty$'' can be essentially enlarged within the class of Orlicz spaces. For example, if $B(t)$ is a Young function that is equivalent to
\begin{equation*}
\begin{cases}
e^{-t^{\beta_0}}\quad\text{for some $\beta_0<0$ or}\\
0
\end{cases}
\end{equation*}
near zero, or equivalent to
\begin{equation*}
\begin{cases}
e^{t^{\beta_\infty}}\quad\text{for some $\beta_\infty>0$ or}\\
\infty
\end{cases}
\end{equation*}
near infinity, then \eqref{thm:optimalorliczdomain:condonB} is satisfied but each Orlicz space $L^A$ that renders \eqref{intr:homoembOrlicz} true can be essentially enlarged to a bigger Orlicz space that still renders \eqref{intr:homoembOrlicz} true.
\end{remark}

We conclude this paper with a reduction principle for the inequality \eqref{intr:homoembOrlicz}. This principle follows from Theorem~\ref{thm:reductionprinciple} and \citep[Theorem~3.3.2, Proposition~3.3.4, Theorem~3.5.2]{VejtekPhD}.
\begin{theorem}\label{thm:reductionprincipleOrlicz}
Assume that $m < n$ and let $A$ and $B$ be Young functions. Then the following four statements are equivalent.
\begin{enumerate}
\item There exists a positive constant $C_1$ such that
\begin{equation*}
\|u\|_{L^B}\leq C_1 \|\nabla^m u\|_{L^A}\quad\text{for each $u\in V_0^m L^A(\rn)$}.
\end{equation*}
\item The Young function $A$ satisfies \eqref{thm:optimalrangerforOrlicz:con0} and there exists a positive constant $C_2$ such that
\begin{equation*}
B(t)\leq A_m(C_2t)\quad\text{for each $t\geq0$},
\end{equation*}
where the Young function $A_m$ is defined by \eqref{thm:optimalorliczrange:defAm}.
\item The Young function $B$ satisfies \eqref{thm:optimalorliczdomain:condonB} and there exists a positive constant $C_3$ such that
\begin{equation*}
\int_0^t\frac{\widetilde{A}(s)}{s^{\frac{n}{n-m}+1}}\,\d s\leq \frac{\widetilde{B_m}(C_3t)}{t^{\frac{n}{n-m}}}\quad\text{for each $t\geq0$},
\end{equation*}
where the Young function $B_m$ is defined by \eqref{thm:optimalorliczdomain:defBm}.
\item There exists a positive constant $C_4$ such that
\begin{equation*}
\int_0^\infty B\left(\frac{\int_t^\infty |f(s)|s^{\frac{m}{n}-1}\,\d s}{C_4\left(\int_0^\infty A(|f(s)|)\,\d s\right)^\frac{m}{n}}\right)\,\d t\leq\int_0^\infty A(|f(t)|)\,\d t\quad\text{for each $f\in L^A(0,\infty)$}.
\end{equation*}
\end{enumerate}
Moreover, the positive constants $C_1$, $C_2$, $C_3$ and $C_4$ depend only on each other, on $m$ and on the dimension $n$.
\end{theorem}
\bibliography{\jobname}

\end{document}